\documentclass[a4paper,12pt]{article}

\usepackage{amsmath,amssymb,amsthm, mathrsfs}
\usepackage{latexsym}
\usepackage{graphics}
\usepackage{hyperref}
\usepackage{graphicx,psfrag}
\usepackage[bf, small]{titlesec}
\usepackage{authblk} %%% new line between authors

\theoremstyle{plain}
\newtheorem{Thm}{Theorem}[section]
\newtheorem{Lem}[Thm]{Lemma}
\newtheorem{Prop}[Thm]{Proposition}
\newtheorem{Cor}[Thm]{Corollary}

\newtheorem{Conj}[Thm]{Conjecture}

\theoremstyle{definition}
\newtheorem{Defi}[Thm]{Definition}
\newtheorem{Rem}[Thm]{Remark}
\newtheorem{Ex}[Thm]{Example}

\usepackage{standalone}
\usepackage{pgfpages,tikz,pgfkeys,pgfplots}
\usetikzlibrary{arrows,positioning,matrix,fit,backgrounds,shapes,shapes.geometric, intersections}

\usetikzlibrary{shapes.callouts,decorations.pathmorphing} %%% 팝업용
\usetikzlibrary{shadows} % shadow
\usepackage{calc} % calculate angles to evenly space vertices in circular arrangements.
\usetikzlibrary{decorations.markings} %%put arrowheads in the middle of directed edges
\tikzstyle{vertex}=[circle, draw, inner sep=0pt, minimum size=6pt] % style
\newcommand{\vertex}{\node[vertex]}
\usetikzlibrary{arrows,matrix} %%% Hesse Diagram

\newcommand{\RR}{\mathbb{R}} %%%%% mathbb
\newcommand{\NN}{\mathbb{N}} %%%%%
\newcommand{\PPP}{\mathcal{P}} %%%%%
\newcommand{\AAA}{\mathcal{A}} %%%%%
\newcommand{\CCC}{\mathcal{C}} %%%%%
\newcommand{\FFF}{\mathcal{F}} %%%%%
\newcommand{\LLL}{\mathcal{L}} %%%%%
\newcommand{\HHH}{\mathcal{H}} %%%%%
\newcommand{\KKK}{\mathcal{K}} %%%%%
\title{A new minimal chordal completion}

\author{
\textsc{Jihoon Choi}%
\footnote{Department of Mathematics Education,
Cheongju University, Cheongju 28503, Korea.
\textit{E-mail}: \texttt{jihoon@cju.ac.kr}}
\quad
\textsc{Soogang Eoh}%
\footnote{Department of Mathematics Education,
Seoul National University, Seoul 08826, Korea.
\textit{E-mail}: \texttt{mathfish@snu.ac.kr}}
\quad
\textsc{Suh-Ryung Kim}%
\footnote{Department of Mathematics Education,
Seoul National University, Seoul 08826, Korea.
\textit{E-mail}: \texttt{srkim@snu.ac.kr}}
}

\begin{document}
\maketitle
\begin{abstract}
In this paper, we present a minimal chordal completion $G^*$ of a graph $G$ satisfying the inequality $\omega(G^*) - \omega(G) \le i(G)$ for the non-chordality index $i(G)$ of $G$.
In terms of our chordal completions, we partially settle the Hadwiger conjecture and the Erd\H{o}s-Faber-Lov\'{a}sz Conjecture, and extend the known $\chi$-bounded class by adding to it the family of graphs with bounded non-chordality indices.
\end{abstract}


\section{Introduction}
Throughout this paper, we only consider simple graphs. We mean by a \emph{clique} of a graph both the vertex set of a complete subgraph and the complete subgraph itself.
The \emph{clique number} of a graph $G$ is defined to be the number of vertices in a maximum clique and denoted by $\omega(G)$. A \emph{hole} of a graph $G$ is an induced cycle of length at least $4$ in $G$. A \emph{chordal graph} is a graph without hole.

A graph $H$ is called a \emph{chordal completion} of a graph $G$, if $H$ is a chordal spanning supergraph of $G$.
See~\cite{heggernes2006minimal} for a survey on chordal completion.
In this paper, we present  a chordal completion of a graph $G$ which is efficient in the following sense:
Only edges joining two vertices on holes of $G$ are added to obtain our chordal completion (Theorem~\ref{thm:lco}).
Furthermore, we show that any minimal chordal completion of a graph can be obtained by joining two vertices on holes of $G$ by edges (Proposition~\ref{prop:local}).
As a matter of fact, for a nonnegative integer $k$, we give a sufficient condition for a graph $G$ which has a chordal completion $G^*$ satisfying the inequality $\omega(G^*) - \omega(G) \le k$ (Theorem~\ref{thm:lco}).
This is a strong point of our chordal completion which differentiate it from other chordal completions. For example, it is shown that a graph $G$ has treewidth at most $k$ if and only if it has a chordal completion $G^*$ satisfying $\omega(G^*) \le k+1$.
Yet, this characterization gives no information on  $\omega(G)$, accordingly no significant information on $\omega(G^*) - \omega(G)$.

The notion ``NC property''  plays a key role throughout this paper.
We say that a hole $H$ \emph{contains} a vertex $v$ (resp.\ an edge $e$) if $v$ (resp.\ $e$) is a vertex (resp.\ an edge) on $H$.
We denote the set of holes in a graph $G$ by $\mathcal{H}(G)$ and the set of holes in $G$ containing $u$ by $\HHH(G, u)$.

A nonempty subset $X$ of $V(G)$ is called a \emph{hole cover} of $G$ provided that every hole in $G$ contains at least one vertex of $X$.
Note that, if $G$ has no hole, that is, $G$ is a chordal graph, then any nonempty vertex set is a hole cover of $G$.

For a vertex $u$ of a graph $G$, we say that $u$ satisfies the \emph{non-consecutive property} (\emph{NC property} for short) if any hole in $\HHH(G, u)$ and any hole not in $\HHH(G, u)$ do not share consecutive edges.
A vertex subset $\CCC$ of $G$ is said to satisfy the \emph{NC property in $G$} if every vertex in $\CCC$ satisfies the NC property and every hole in $G$ contains at most one vertex in $\CCC$.
We say that a graph satisfies the \emph{NC property} if it has a hole cover satisfying the NC property.
It is easy to see that
\begin{itemize}
  \item[($\natural$)] If a hole cover $\CCC$ of $G$ satisfies the NC property in $G$, then a nonempty set $\CCC \setminus \AAA$ is a hole cover satisfying the NC property in $G - \AAA$ for any proper subset $\AAA$ of $V(G)$ not including $\CCC$.
\end{itemize}
Then it is immediately true that any induced subgraph of a graph satisfying the NC property satisfies the NC property.
See Figure~\ref{fig:noneg} for a graph not satisfying the NC property.
To see why, suppose to the contrary that there exists a hole cover $\CCC$ of $G$ with the NC property.
To cover the hole $H_2$, $\CCC$ must contain a vertex on $H_2$.
Suppose that a vertex in $V(H_1) \cap  V(H_2)$ is contained in $\CCC$.
Since $\CCC$ is a hole cover satisfying the NC property, a vertex in $V(H_3)\setminus V(H_2)$ must be contained in $\CCC$ to cover $H_3$.
Then, however, those two vertices are on the hole of length $8$ surrounding $H_2$ and $H_3$, which contradicts the assumption that $\CCC$ satisfies the NC property.
Even if a vertex in $V(H_2) \cap  V(H_3)$ is contained in $\CCC$, we may reach a contradiction by applying a similar argument to the holes $H_1$ and $H_2$.
Therefore we may conclude that there is no hole cover of $G$ satisfying the NC property.
If we delete the vertex $x$, we obtain a graph satisfying the NC property as $\{u,v,w\}$ is a hole cover of the new graph with the NC property.

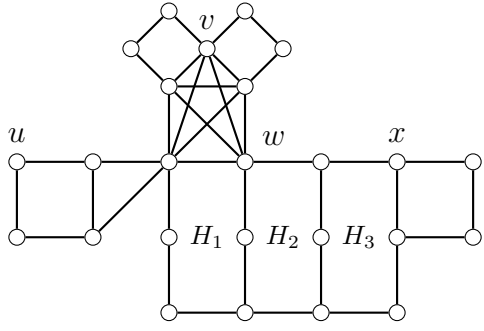
\begin{figure}
  \centering

\begin{tikzpicture}[x=1cm, y=1cm]

    \draw (2.5,-1) node{\footnotesize $H_1$};
    \draw (3.5,-1) node{\footnotesize $H_2$};
    \draw (4.5,-1) node{\footnotesize $H_3$};

    \vertex (a1) at (0,0) [label=above:$u$]{};
    \vertex (a2) at (0,-1) [label=above:$$]{};
    \vertex (a3) at (1,-1) [label=above:$$]{};
    \vertex (a4) at (1,0) [label=above:$$]{};

    \vertex (b1) at (2,0) [label=above:$$]{};
    \vertex (b2) at (2,-1) [label=above:$$]{};
    \vertex (b3) at (2,-2) [label=above:$$]{};

    \vertex (b4) at (3,0) [label=above right:$w$]{};
    \vertex (b5) at (3,-1) [label=right:$$]{};
    \vertex (b6) at (3,-2) [label=right :$$]{};

    \vertex (b7) at (4,0) [label=above:$$]{};
    \vertex (b8) at (4,-1) [label=above:$$]{};
    \vertex (b9) at (4,-2) [label=above:$$]{};

    \vertex (b10) at (5,0) [label=above:$x$]{};
    \vertex (b11) at (5,-1) [label=above:$$]{};
    \vertex (b12) at (5,-2) [label=above:$$]{};

    \vertex (b13) at (6,0) [label=above:$$]{};
    \vertex (b14) at (6,-1) [label=above:$$]{};

    \vertex (c1) at (2,1) [label=above:$$]{};
    \vertex (c2) at (1.5,1.5) [label=above:$$]{};
    \vertex (c3) at (2,2) [label=above:$$]{};
    \vertex (c4) at (2.5,1.5) [label=above:$v$]{};
    \vertex (c5) at (3,2) [label=above:$$]{};
    \vertex (c6) at (3.5,1.5) [label=above:$$]{};
    \vertex (c7) at (3,1) [label=above:$$]{};

    \path
    (a1) edge [-,thick] (a2)
    (a2) edge [-,thick] (a3)
    (a3) edge [-,thick] (a4)
    (a4) edge [-,thick] (a1)

    (b1) edge [-,thick] (a3)
    (b1) edge [-,thick] (a4)

    (b1) edge [-,thick] (b2)
    (b2) edge [-,thick] (b3)

    (b4) edge [-,thick] (b5)
    (b5) edge [-,thick] (b6)

    (b7) edge [-,thick] (b8)
    (b8) edge [-,thick] (b9)

    (b10) edge [-,thick] (b11)
    (b11) edge [-,thick] (b12)

    (b13) edge [-,thick] (b14)

    (b1) edge [-,thick] (b4)
    (b4) edge [-,thick] (b7)
    (b7) edge [-,thick] (b10)
    (b10) edge [-,thick] (b13)

%    (b2) edge [-,thick] (b5)
%    (b5) edge [-,thick] (b8)
%    (b8) edge [-,thick] (b11)
    (b11) edge [-,thick] (b14)

    (b3) edge [-,thick] (b6)
    (b6) edge [-,thick] (b9)
    (b9) edge [-,thick] (b12)

    (c1) edge [-,thick] (b1)
    (c1) edge [-,thick] (b4)
    (c4) edge [-,thick] (b1)
    (c4) edge [-,thick] (b4)
    (c7) edge [-,thick] (b1)
    (c7) edge [-,thick] (b4)

    (c1) edge [-,thick] (c2)
    (c2) edge [-,thick] (c3)
    (c3) edge [-,thick] (c4)
    (c4) edge [-,thick] (c1)
    (c4) edge [-,thick] (c5)
    (c5) edge [-,thick] (c6)
    (c6) edge [-,thick] (c7)
    (c7) edge [-,thick] (c1)
    (c7) edge [-,thick] (c4)

	;

\end{tikzpicture}

  \caption{A graph $G$ not satisfying the NC property }\label{fig:noneg}
\end{figure}

Let $G$ be a graph with a hole $H$.
For a vertex $u$ on $H$, we \emph{locally chordalize} the hole $H$ by $u$ in the following manner: we join $u$ and each vertex on $H$ nonadjacent to $u$.

For a graph $G$ and a vertex $u$ satisfying the NC property, locally chordalizing all the holes in $\HHH(G, u)$ does not create any new hole (Theorem~\ref{thm:chordalize}).
Based on this observation, we found that, a hole cover $\CCC$ of a graph $G$ can be partitioned into $\CCC_1, \ldots, \CCC_k$ for some positive integer $k$ so that
\begin{itemize}
  \item[(i)] $\CCC_i$ is a hole cover of the graph $G_{i}$ satisfying the NC property,
  \item[(ii)] $G_{i}^*$ is chordal,
\end{itemize}
where $G_0 = G_0^* = G - \CCC$, $G_{i}$ is the graph defined by $V(G_i) = V(G_{i-1}^*) \cup \CCC_{i}$ and
      \[
         E(G_i) =
         %E(G_{i-1}^*) \cup \{vw \in E(G) \mid v \in \CCC_i, w \in V(G_{i-1}^*)\} ( =
         E(G_{i-1}^*) \cup E \left( G - \bigcup_{j=i+1}^k \CCC_j\right),
   \]
and $G_i^*$ is a chordal completion of $G_i$ obtained by applying local chordalizations recursively by the vertices in $\CCC_i$ for each $i=1$, $\ldots$, $k$ (Theorem~\ref{thm:existence}).
Our chordal completion is $G_k^*$ obtained for a hole cover with the smallest number $k$ of partitions in Theorem~\ref{thm:existence}.
The smallest number $k$ is called the non-chordality index of $G$ and denoted by $i(G)$ (see Definition~\ref{def:non-chordality index}).

% the chordal completion $G^*$ of $G$ presented in this paper satisfies the inequality $\omega(G^*) - \omega(G) \le 1$ if $G$ satisfies the NC property (Corollary~\ref{cor:newclique}).

%
%We show that any minimal chordal completion of a graph can be obtained by joining two vertices on holes of $G$ by edges (Proposition~\ref{prop:local}).
%Then, we construct a minimal chordal completion by joining two vertices on holes of $G$ by edges, and, by utilizing this, we give a sufficient condition so called the ``NC property'' for the existence of a chordal completion $G^*$ of a graph $G$ satisfying $\omega(G^*) - \omega(G) \le 1$.

For a given graph $G$, a function $f : V (G) \to \{1, 2,  \ldots, k \}$ is called a \emph{proper $k$-coloring} of $G$ if $f(u) \neq f(v)$ for any adjacent vertices $u$ and $v$.
The \emph{chromatic number} $\chi(G)$ of a graph G is defined to be the least positive integer $k$ such that there exists a proper $k$-coloring of $G$.
It is well-known that, for a graph $G$,
\begin{equation}\label{eqn:chi,l,dp}
\chi(G) \le \chi_l(G) \le \chi_{DP}(G)
\end{equation}
where $\chi_l(G)$ and $\chi_{DP}(G)$ are the list-chromatic number and the DP-chromatic number of $G$, respectively (see \cite{vizing1976vertex}, \cite{erdos1979choosability}, \cite{dvovrak2017correspondence} for the definitions).

For a positive integer $k$, a graph $G$ is \emph{$k$-degenerate} if any subgraph of $G$ contains a vertex having at most $k$ neighbors in it. Dvo\v{r}\'{a}k and Postle~\cite{dvovrak2017correspondence} observed that if a graph $G$ is $k$-degenerate, then
$\chi_{DP}(G) \le k+1$.
It is easy to check that every chordal graph $G$ is $(\omega(G)-1)$-degenerate and so
\begin{equation*}\label{eqn:chordal}
\omega(G) \le \chi(G) \le \chi_l(G) \le \chi_{DP}(G) \le \omega(G).
\end{equation*}
Therefore,
\begin{itemize}
  \item[($\S$)] for a chordal graph $G$, $\chi(G)=\chi_l(G)=\chi_{DP}(G)=\omega(G)$.
\end{itemize}

The observation ($\S$) directed our attention to the idea that, for a chordal completion $G^*$ of a graph $G$, the chromatic number of $G$ is bounded above by the clique number of $G^*$.

By \eqref{eqn:chi,l,dp}, an upper bound of $\chi_{DP}(G)$ (resp.\ $\chi_l(G)$) is an upper bound of $\chi_l(G)$ (resp.\ $\chi(G)$).
In this vein, it is interesting to check whether or not $\chi_{DP}(G) \le k$ (resp.\ $\chi_l(G) \le k$) when $\chi_l(G) \le k$ (resp.\ $\chi(G) \le k$) for a positive integer $k$.

%It is known that there are infinitely many simple graph $G$ satisfying $\chi_l(G) < \chi_{DP}(G)$ and the gap between $\chi_l(G)$ and $\chi_{DP}(G)$ can be arbitrarily large.

We obtain sharp upper bounds for the chromatic number, the list chromatic number, and the DP chromatic number of a graph in terms of non-chordality index and partially settle the Hadwiger conjecture (Theorems~\ref{thm:coloring1}, \ref{thm:coloring2}, and \ref{thm:omega+i is sharp}).
Other than obtaining sharp upper bounds for chromatic numbers, we partially settle the Erd\H{o}s-Faber-Lov\'{a}sz Conjecture (Theorem~\ref{thm:erdos}), and extend the known $\chi$-bounded class by adding to it the family of graphs with bounded non-chordality indices. (Theorem~\ref{thm:chibound}).

%
%we shall introduce the notions of NC property and $i(G)$ of a graph $G$ and show the following two results.
%
%
%
%
%As a corollary of Corollary~\ref{cor:coloring1}, we will obtain a special case of Four Color Theorem that every planar graph $G$ with the NC property satisfies $\chi(G) \le 4$.

\section{Graphs with the NC property}

%It is easy to check that any hole in $G$ contains at most one vertex in $\CCC$ satisfying the NC property.
In this section, we devote ourselves to proving the following theorem.

\begin{Thm}\label{thm:coloring1}
Let $G$ be a graph with the NC property.
Then $\chi_{DP}(G) \le \omega(G) + 1$.
If $G$ is $K_n$-minor-free, then $\chi_{DP}(G) \le n-1$.
\end{Thm}

%
%\begin{Thm}\label{thm:knminorfree}
%Let $G$ be a graph with the NC property.
%Then $\chi(G) \le \omega(G) + 1$.
%If $G$ is $K_n$-minor-free, then $\chi(G) \le n-1$.
%\end{Thm}
%\begin{proof}
%Let $\CCC$ be a hole cover of $G$ having the NC property.
%Since $G$ is a spanning subgraph of $\widehat{G}(\CCC)$, $\chi(G) \le \chi(\widehat{G}(\CCC))$.
%Since $\widehat{G}(\CCC)$ is a chordal graph, it is perfect and so $\chi(\widehat{G}(\CCC)) = \omega(\widehat{G}(\CCC))$.
%By Theorem~\ref{thm:newclique}, $\omega(\widehat{G}(\CCC)) \le \omega(G) + 1$.
%Thus $\chi(G) \le \omega(G) + 1$.
%If $G$ is $K_n$-minor-free, then $\omega(\widehat{G}(\CCC)) \le n-1$ by Corollary~\ref{cor:knminor} and so $\chi(G) \le n-1$.
%\end{proof}
\noindent
As a corollary of Theorem~\ref{thm:coloring1}, we can prove a special case of Four Color Theorem.

\begin{Cor}\label{cor:planar}
For a planar graph $G$ with the NC property, $\chi_{DP}(G) \le 4$.
\end{Cor}

\begin{Lem}\label{lem:chord}
Given a graph $G$ and a cycle $C$ of $G$ with length at least four,
suppose that a section $Q$ of $C$ forms an induced path of $G$ and contains a path $P$  with length at least two none of whose internal vertices is incident to a chord of $C$ in $G$.
Then $P$ can be extended to a hole $H$ in $G$ so that $V(P) \subsetneq V(H) \subset V(C)$ and $H$ contains a vertex on $C$ not on $Q$.
\end{Lem}
\begin{proof}
Let $v_i$ and $v_j$ be the origin and the terminal of $P$.
Since $P$ is an induced path of length at least two, $v_i$ and $v_j$ are nonadjacent.
Now we take a shortest $(v_j,v_i)$-path $P'$ with some  vertices on the $(v_i,v_j)$-section of $C$ other than $P$.
Since $v_i$ and $v_j$ are nonadjacent, $P'$ has length at least two.
Therefore $PP'$ is a cycle of length at least four.
By the hypothesis, none of the internal vertices of $P$ is incident to a chord of $C$.
In addition, $P$ and $P'$ are induced paths, so $H:=PP'$ is actually a hole in $G$.
Note that $V(H) \subset V(C)$.
If every vertex on $H$ were on $Q$, then $Q$ would have a chord as $V(H) \subset V(Q)$, which is impossible.
Therefore $H$ contains a vertex on $C$ not on $Q$.
\end{proof}

%\begin{Lem}
%For a graph $G$, let $H_u = uu_1u_2 \cdots u_p w u_{p+1} \cdots u_qu$ ($1 \le p <q$) and $H_v = vv_1v_2 \cdots v_r w v_{r+1} \cdots v_sv$ ($1 \le r <s$) be holes in $G$.
%If there is no hole containing both $u$ and $v$, then $\{ 1 \le i \le r \mid \text{$uv_i$ is an edge in $G$} \} = \emptyset$ or $\{ r+1 \le i \le s \mid \text{$uv_i$ is an edge in $G$} \} = \emptyset$ (and  $\{ 1 \le j \le p \mid        \text{$u_jv$ is an edge in $G$} \} = \emptyset$ or $\{ p+1 \le i \le q \mid \text{$u_jv$ is an edge in $G$} \} = \emptyset$).
%\end{Lem}
Given a graph $G$ and nonempty vertex sets $S_1$ and $S_2$, we denote the set of edges joining vertices of $S_1$ and vertices of $S_2$ by $[S_1, S_2]$.
For simplicity, we use $[v, S]$ instead of $[\{v\}, S]$ for a vertex $v$ and a nonempty vertex set $S$ of a graph $G$.

\begin{Lem}\label{lem:hole}
Given a graph $G$, suppose that there exist a hole $H$, an induced path $P$, and two nonadjacent vertices $u$ and $v$ on $H$ not on $P$ satisfying the properties that
\begin{itemize}
  %\item[(i)] there exist two nonadjacent vertices $u$ and $v$ on $H$ but not on $P$;
  \item[(i)] $v$ is nonadjacent to any vertex on $P$ in $G$;
  \item[(ii)] there exist an internal vertex on a $(u,v)$-section of $H$ and an internal vertex on the other $(u,v)$-section of $H$ such that each of them is adjacent to a vertex on $P$.
\end{itemize}
Then there is a hole not containing $u$ but containing two consecutive edges on $H$ incident to $v$ and containing a vertex on $P$ but not on $H$.
\end{Lem}
\begin{proof}
Let $P = z_1 z_2 \cdots z_r$ ($r \ge 1$).
By the hypothesis that $u$ and $v$ are not on $P$, $z_i \neq u,v$ for each $i=1,\ldots,r$.
Since $u$ and $v$ are nonconsecutive vertices on $H$, we may give a sequence of $H$ as follows:
\[
H = v x_1 x_2 \cdots x_p u y_q y_{q-1} \cdots y_1 v \ (p,q \ge 1).
\]
For notational convenience, we let $S_x = \{x_1, \ldots, x_p\}$ and $S_y = \{y_1, \ldots, y_q\}$.
Let $\alpha = \min \{i \in \{1,\ldots,p\} \mid [x_i, V(P)] \neq \emptyset  \}$ and $\beta = \min \{j \in \{1,\ldots,q\} \mid [y_j, V(P)] \neq \emptyset \}$.
By the property (ii), $[S_x, V(P)] \neq \emptyset$ and $[S_y, V(P)] \neq \emptyset$ and so $\alpha$ and $\beta$ exist.
Among the vertices on $P$ which are adjacent to $x_\alpha$ and among the vertices on $P$ which are adjacent to $y_\beta$, we take $z_\gamma$ and $z_\delta$ from them, respectively, with the smallest distance on $P$.
%Now we take $z_\gamma, z_\delta \in V(P)$ ($z_\gamma$ and $z_\delta$ may coincide) such that $x_\alpha z_\gamma, y_\beta  z_\delta \in E(G)$ and $d_P(z_\gamma, z_\delta)$ becomes as small as possible.
Let $P^*$ be the $(z_\gamma, z_\delta)$-section of $P$.
Then $C := v x_1 x_2 \cdots x_\alpha P^* y_\beta y_{\beta-1} \cdots y_1 v$ is a cycle not containing $u$.
We also note that $C$ contains $x_1v$ and $y_1v$, which are consecutive edges on $H$ incident to $v$.
It is easy to check that $C$ has length at least four.
No two vertices in $V(C) \setminus V(P^*)$ or in $V(P^*)$ can form a chord of $C$ since the vertices in $V(C) \setminus V(P^*)$ are on the hole $H$ and $P^*$ is an induced path.
Moreover, a vertex in $V(C) \setminus V(P^*)$ and a vertex in $V(P^*)$ cannot form a chord of $C$ by the choice of $\alpha$, $\beta$, $z_\gamma$, and $z_\delta$.
Therefore we can conclude that $C$ is a hole in $G$.
Since $u$ is not on $C$, $C$ is distinct from $H$.
We note that $C$ and $H$ both are holes and the vertices on $C$ other than the ones on $P^*$ lie on $H$.
Therefore there must be a vertex on $P^*$ not on $H$.
Since $P^*$ is a section of $P$, $C$ contains a vertex on $P$ but not on $H$.
\end{proof}

%\begin{Cor}\label{cor:hole}
%Given a graph $G$ and a hole $H$ of $G$, suppose that there exist a vertex $v$ on $H$ and a vertex $w$ not on $H$ such that $v$ and $w$ are nonadjacent and $w$ is adjacent to two nonconsecutive vertices on $H$.
%Then there exists a hole containing $w$ and two consecutive edges of $H$ incident to $v$.
%\end{Cor}
%\begin{proof}
%Let $x$ and $y$ be nonconsecutive vertices on $H$ which are adjacent to $w$.
%Then the $(x,y)$-section of $H$ not containing $v$ has length at least two.
%We take an internal vertex $u$ of the section.
%It is easy to check that the hole $H$, the induced path $P:=w$, and the vertices $u$ and $v$ satisfy the properties (i) and (ii) in Lemma~\ref{lem:hole}.
%\end{proof}

Let $G$ be a graph with a hole $H$.
For a vertex $u$ on $H$, we recall that locally chordalizing the hole $H$ by $u$ means the following procedure: we join $u$ and each vertex on $H$ nonadjacent to $u$.
We call an edge added in the process of a local chordalization of a hole a \emph{newly added edge}.

\begin{Rem}
Note that, for a graph $G$, locally chordalizing the holes in $\HHH(G, u)$ by a vertex $u$ will destroy all the holes in $\HHH(G, u)$.
That is, if $H \in \HHH(G, u)$, then $H \notin \HHH(G^*, u)$ where $G^*$ is the graph resulting from the local chordalization by $u$.
%Theorem~\ref{thm:chordalize}
\end{Rem}

\begin{Thm}\label{thm:chordalize}
Let $G$ be a graph and $u$ be a vertex of $G$ satisfying the NC property.
Then locally chordalizing all the holes in $\mathcal{H}(G, u)$ by $u$ does not create any new hole.
\end{Thm}
\begin{proof}
Let $G^*$ be the graph obtained by locally chordalizing all the holes in $\HHH(G, u)$ by $u$.
Suppose to the contrary that $G^*$ has a hole, say $H^*$, not in $G$.
Obviously $H^*$ contains $u$ and at least one newly added edge.
Then, since $u$ is adjacent to exactly two vertices on $H^*$, $H^*$ contains one newly added edge or two newly added edges.

Let $uv$ be a newly added edge and
\[
H^* = uu_1u_2\cdots u_p v u \quad (p \ge 2).
\]
Next, we define a cycle $C$ by considering two cases.

{\it Case 1}. $H^*$ contains $uv$ as the only newly added edge.
By the definition of local chordalization, there exists a hole $H_1$ in $\HHH(G, u)$ containing $v$ on which $u$ and $v$ are not consecutive.
%Let $H = ux_1x_2\cdots x_q v y_r y_{r-1} \cdots y_1u$ ($q,r \ge 1$).
Then $u$ is adjacent to all the vertices on $H_1$ in $G^*$.
However, $u$ is not adjacent to $u_k$ ($k=2, \ldots, p)$ in $G^*$, so we can conclude that $u_k$ ($k=2,\ldots,p$) is not on $H_1$.
If $u_1$ is on $H_1$, then $u_1$ is adjacent to $u$ in $H_1$.
Thus, if $u_1$ is on $H_1$, then $uu_1u_2 \cdots u_p P$ is a cycle in $G$ for the  $(v,u)$-section, denoted by $P$, of $H_1$ not containing $u_1$.

If $u_1$ is not on $H_1$ and $u_1$ is not adjacent to any vertex on one of the $(v,u)$-sections of $H_1$ except $u$, then we denote such a section by $P'$.

Now we define the cycle  $C$ as follows:
\[
C=
\begin{cases}
   uu_1u_2 \cdots u_p P& \parbox{7cm}{if $u_1$ is on $H_1$;} \\[10pt]

   u u_1 u_2 \cdots u_p P' & \parbox{7cm}{if $u_1$ is not on $H_1$ and $u_1$ is not adjacent to any vertex on one of the $(v,u)$-sections of $H_1$ except $u$.}
\end{cases}
\]
See (a) and (b) of Figure~\ref{fig:example} for an illustration.

{\it Case 2}. $H^*$ contains another newly added edge $uw$.
Then $u_1 = w$.
Assume that there is a hole $H_2$ in $G$ which contains $u, v, w$.
Then no two of $u, v, w$ are consecutive on $H_2$.
Let $Q$ be the $(v,w)$-section of $H_2$ containing $u$.
Since $u$ is adjacent to all the vertices on $H_2$ but is not adjacent to $u_i$ in $G^*$,
we may conclude that $u_i$ is not on $H_2$ for each $i=2,\ldots,p$.
Therefore $w u_2 u_3 \cdots u_p Q$ is a cycle in $G$.
Now we let
\[
C = w u_2 u_3 \cdots u_p Q.
\]
See Figure~\ref{fig:example}(c) for an illustration.

It is obvious that the cycle $C$ defined in each case has length at least four and $Pu_1$, $P'u_1$, and $Q$ are induced paths of $G$ including $u$ and the two vertices right next to $u$ on $C$.
%Let $x$ and $x'$ be the vertices which immediately proceed on $P$ and $P'$, respectively, and $y$ and $z$ the vertices on $Q$ which are adjacent to $u$
Moreover, $u$ is not adjacent to any vertex on $C$ except the two vertices right next to $u$, and the two vertices right next to $u$ on $C$ are not adjacent in $G$.
Thus, by Lemma~\ref{lem:chord}, the path $U$ composed of $u$ and the two vertices right next to it can be extended to a hole $H$ in $G$ so that $V(U) \subsetneq V(H) \subset V(C)$ and $H$ contains a vertex among $u_2, u_3, \ldots, u_p$.
Then $u$ is adjacent to $u_i$ for some $i \in \{2,3,\ldots,p\}$ in $G^*$ by the definition of local chordalization, which contradicts the choice of $H^*$.

Now it remains to consider the following cases:
\begin{itemize}
    \item[(i)] the edge $uv$ is the only newly added edge contained in $H^*$, $u_1$ is not on $H_1$, and there is a vertex on each $(u,v)$-section of $H_1$ which is adjacent to $u_1$ in $G$;
    \item[(ii)] a newly added edge $uw$ other than $uv$ exists in $H^*$ and there is no hole in $G$ which contains all of $u$, $v$, and $w$.
\end{itemize}

We assume the case (i).
The hypothesis of Lemma~\ref{lem:hole} is satisfied by $H_1$ for $H$, $u_1$ for $P$, $u$, and $v$.
Therefore there exists a hole not containing $u$ but containing consecutive edges on $H_1$ incident to $v$.
This contradicts the hypothesis that $u$ satisfies the NC property.
Therefore the case (i) cannot happen.

Now we assume the case (ii).
Since $v$ and $w$ are not consecutive vertices on $H^*$, $w$ is not adjacent to $v$ in $G$.
Since $uv$ and $uw$ are newly added edges, there exist a hole $H_3$ containing $u$ and $v$, and a hole $H_4$ containing $u$ and $w$ in $G$.
By the case (ii) assumption, $w$ is not on $H_3$ and $v$ is not on $H_4$.
Let $H_3 = v x_1 x_2 \cdots x_q u y_r y_{r-1} \cdots y_1 v$ and $H_4 = w z_1 z_2 \cdots z_s u w_t w_{t-1} \cdots w_1 w$ ($q,r,s,t \ge 1$).
See Figure~\ref{fig:example}(d) for an illustration.
Since $u$ is adjacent to all the vertices on $H_3$ (resp.\ $H_4$) and is not adjacent to $u_i$ in $G^*$,
we may conclude that $u_i$ is not on $H_3$ (resp.\ $H_4$) for each $i=2, \ldots, p$.
For notational convenience, we let $S_x = \{x_1, \ldots, x_q\}$, $S_y = \{y_1, \ldots, y_r\}$, $S_z = \{z_1, \ldots, z_s\}$, and $S_w = \{w_1, \ldots, w_t\}$.

Suppose that, in $G$, $[w,S_x] \neq \emptyset$ and $[w,S_y] \neq \emptyset$.
We apply Lemma~\ref{lem:hole} with $H_3$ for $H$, $w$ for $P$, $u$, and $v$ to reach a contradiction as before.
Therefore $[w,S_x] = \emptyset$ or $[w,S_y] = \emptyset$.
Without loss of generality, we may assume $[w,S_x] = \emptyset$.
In addition, $w$ is not adjacent to $v$ in $G$.
Thus $[w,S_x \cup \{v\}] = \emptyset$.

Suppose that $[S_x \cup \{v\}, S_z] \neq \emptyset$ and $[S_x \cup \{v\}, S_w] \neq \emptyset$.
Then we apply Lemma~\ref{lem:hole} with $H_4$ for $H$, $vx_1x_2\cdots x_q$ for $P$, $u$, and $w$ for $v$ to reach a contradiction as before.
Therefore $[S_x \cup \{v\}, S_z] = \emptyset$ or $[S_x \cup \{v\}, S_w] = \emptyset$.
Without loss of generality, we may assume $[S_x \cup \{v\}, S_z] = \emptyset$.
Then $[S_x \cup \{v\}, S_z \cup \{w\}] = \emptyset$.

Now we consider the sequence $Q := v x_1x_2\cdots x_q u z_s z_{s-1} \cdots z_1 w$.
As being sections of $H_3$ and $H_4$, respectively, the two subsequences $v x_1x_2\cdots x_q u$ and $u z_s z_{s-1} \cdots z_1 w$ of $Q$ are induced paths in $G$.
In addition, since $[S_x \cup \{v\}, S_z \cup \{w\}] = \emptyset$, $Q$ is an induced path in $G$.
Consider the cycle $C:=Q u_2 u_{3} \cdots u_p v$.
Since $u$ is on $H_3$, $H_4$, and $H^*$, $u$ is not incident to any chord of $C$ in $G$.
Then we apply Lemma~\ref{lem:chord} with $C$, $Q$, and $x_q u z_s$ for $P$ to reach a contradiction as before.
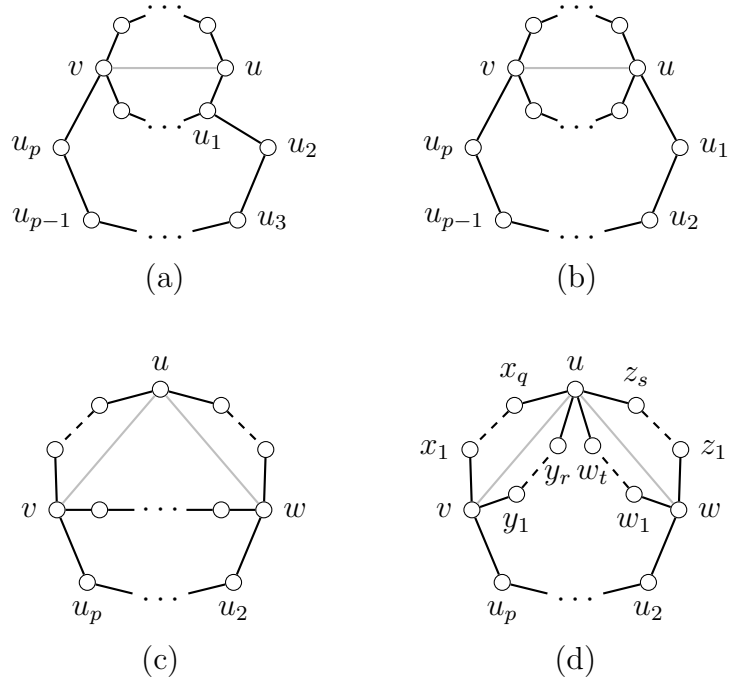
\begin{figure}[t]
\begin{center}
\begin{tikzpicture}[x=0.8cm, y=0.8cm]

    \vertex (u) at (0:1) [label=right:$u$]{};
    \vertex (z2) at (360*1/8:1) [label=above:$$]{};
    %\vertex (z3) at (360*2/8:1) [label=below:$\cdots$]{};
    \node[style={circle,minimum size=5pt}] (z3) at (360*2/8:1) [label=center:\hskip0.165em $\cdots$]{};
    \vertex (z4) at (360*3/8:1) [label=above:$$]{};
    \vertex (v) at (360*4/8:1) [label=left:$v$]{};
    \vertex (z6) at (360*5/8:1) [label=below:$$]{};
    %\vertex (z7) at (360*6/8:1) [label=above:$\cdots$]{};
    \node[style={circle,minimum size=5pt}] (z7) at (360*6/8:1) [label=center:\hskip0.165em $\cdots$]{};
    \vertex (z8) at (360*7/8:1) [label=below:$u_1$]{};

    \begin{scope}[yshift=-30]
    \vertex (u4) at (45*4:1.7) [label=left:$u_p$]{};
    \vertex (u5) at (45*5:1.7) [label=left:$u_{p-1}$]{};
    \vertex (u7) at (45*7:1.7) [label=right:$u_3$]{};
    \vertex (u8) at (45*8:1.7) [label=right:$u_2$]{};

    %\vertex[style={circle,fill=black!20}]   (u6) at (45*6:1.5) [label=above:$\cdots$]{};
    \node[style={circle,minimum size=5pt}] (u6) at (45*6:1.5) [label=center:\hskip0.165em $\cdots$]{};
    \end{scope}

    \draw[-, thick] (u) -- (z2);
    \draw[-, thick, shorten >=2] (z2) -- (z3);
    \draw[-, thick, shorten <=2] (z3) -- (z4);
    \draw[-, thick] (z4) -- (v);
    \draw[-, thick] (v) -- (z6);
    \draw[-, thick, shorten >=2] (z6) -- (z7);
    \draw[-, thick, shorten <=2] (z7) -- (z8);
    \draw[-, thick] (z8) -- (u);

    \draw[-, thick] (v) -- (u4);
    \draw[-, thick] (u4) -- (u5);
    \draw[-, thick, shorten >=5] (u5) -- (u6);
    \draw[-, thick, shorten <=5] (u6) -- (u7);
    \draw[-, thick] (u7) -- (u8);
    \draw[-, thick] (u8) -- (z8); % changed edge

    \draw[-, thick, color=lightgray] (v) -- (u);

    \draw (0,-3.5) node{(a)};
\end{tikzpicture}
\qquad \
\begin{tikzpicture}[x=0.8cm, y=0.8cm]

    \vertex (u) at (0:1) [label=right:$u$]{};
    \vertex (z2) at (360*1/8:1) [label=above:$$]{};
    %\vertex (z3) at (360*2/8:1) [label=below:$\cdots$]{};
    \node[style={circle,minimum size=5pt}] (z3) at (360*2/8:1) [label=center:\hskip0.165em $\cdots$]{};
    \vertex (z4) at (360*3/8:1) [label=above:$$]{};
    \vertex (v) at (360*4/8:1) [label=left:$v$]{};
    \vertex (z6) at (360*5/8:1) [label=below:$$]{};
    %\vertex (z7) at (360*6/8:1) [label=above:$\cdots$]{};
    \node[style={circle,minimum size=5pt}] (z7) at (360*6/8:1) [label=center:\hskip0.165em $\cdots$]{};
    \vertex (z8) at (360*7/8:1) [label=below:$$]{};

    \begin{scope}[yshift=-30]
    \vertex (u4) at (45*4:1.7) [label=left:$u_p$]{};
    \vertex (u5) at (45*5:1.7) [label=left:$u_{p-1}$]{};
    \vertex (u7) at (45*7:1.7) [label=right:$u_2$]{};
    \vertex (u8) at (45*8:1.7) [label=right:$u_1$]{};

    %\vertex[style={circle,fill=black!20}]   (u6) at (45*6:1.5) [label=above:$\cdots$]{};
    \node[style={circle,minimum size=5pt}] (u6) at (45*6:1.5) [label=center:\hskip0.165em $\cdots$]{};
    \end{scope}

    \draw[-, thick] (u) -- (z2);
    \draw[-, thick, shorten >=2] (z2) -- (z3);
    \draw[-, thick, shorten <=2] (z3) -- (z4);
    \draw[-, thick] (z4) -- (v);
    \draw[-, thick] (v) -- (z6);
    \draw[-, thick, shorten >=2] (z6) -- (z7);
    \draw[-, thick, shorten <=2] (z7) -- (z8);
    \draw[-, thick] (z8) -- (u);

    \draw[-, thick] (v) -- (u4);
    \draw[-, thick] (u4) -- (u5);
    \draw[-, thick, shorten >=5] (u5) -- (u6);
    \draw[-, thick, shorten <=5] (u6) -- (u7);
    \draw[-, thick] (u7) -- (u8);
    \draw[-, thick] (u8) -- (u);

    \draw[-, thick, color=lightgray] (v) -- (u);

    \draw (0,-3.5) node{(b)};
\end{tikzpicture}
\vskip0.5cm

\hskip0.1cm
%%%%%
\begin{tikzpicture}[x=0.8cm, y=0.8cm]

    \vertex (u) at (90:2) [label=above:$u$]{};

    \vertex (v) at (360*4/8:1.7) [label=left:$v$]{};
    \vertex (u3) at (360*5/8:1.7) [label=below:$u_p$]{};
    %\vertex (u4) at (360*6/8:1.5) [label=left:$z_4$]{};
    \node[style={circle,minimum size=5pt}] (u4) at (360*6/8:1.5) [label=center:\hskip0.165em $\cdots$]{};
    \vertex (u5) at (360*7/8:1.7) [label=below:$u_2$]{};
    \vertex (w) at (360*8/8:1.7) [label=right:${w}$]{};

    \vertex (w1) at (0:1) [label=below:$$]{};
    %\node[style={circle,minimum size=5pt}] (w2) at (40:1) [label=center:\hskip0.165em $\cdots$]{};

    \vertex (z1) at (30:2) [label=right:$$]{};
    \vertex (z3) at (60:2) [label=right:$$]{};

    \vertex (y1) at (180:1) [label=below:$$]{};
    \node[style={circle,minimum size=5pt}] (y2) at (40:0) [label=center:\hskip0.165em $\cdots$]{};

    \vertex (x1) at (150:2) [label=left:$$]{};
    \vertex (x3) at (120:2) [label=left:$$]{};

    \begin{scope}[yshift=-30]

    \end{scope}

    \draw[-, thick, color=lightgray] (u) -- (v);
    \draw[-, thick] (v) -- (u3);
    \draw[-, thick, shorten >=5] (u3) -- (u4);
    \draw[-, thick, shorten <=5] (u4) -- (u5);
    \draw[-, thick] (u5) -- (w);
    \draw[-, thick, color=lightgray] (w) -- (u);

    \draw[-, thick] (u) -- (x3);
    \draw[-, thick, dashed] (x3) -- (x1);
    \draw[-, thick] (x1) -- (v);
    \draw[-, thick] (v) -- (y1);

    \draw[-, thick] (u) -- (z3);
    \draw[-, thick, dashed] (z3) -- (z1);
    \draw[-, thick] (z1) -- (w);
    \draw[-, thick] (w) -- (w1);

    \draw[-, thick, shorten >=4] (w1) -- (y2);
    \draw[-, thick, shorten >=4] (y1) -- (y2);

    \draw (0,-2.5) node{(c)};

\end{tikzpicture}
\qquad \hskip0.25cm
\begin{tikzpicture}[x=0.8cm, y=0.8cm]

    \vertex (u) at (90:2) [label=above:$u$]{};

    \vertex (v) at (360*4/8:1.7) [label=left:$v$]{};
    \vertex (u3) at (360*5/8:1.7) [label=below:$u_p$]{};
    %\vertex (u4) at (360*6/8:1.5) [label=left:$z_4$]{};
    \node[style={circle,minimum size=5pt}] (u4) at (360*6/8:1.5) [label=center:\hskip0.165em $\cdots$]{};
    \vertex (u5) at (360*7/8:1.7) [label=below:$u_2$]{};
    \vertex (w) at (360*8/8:1.7) [label=right:$w$]{};

    \vertex (w1) at (15:1) [label=below:$w_1$]{};
    %\node[style={circle,minimum size=5pt}] (w2) at (40:1) [label=center:\hskip0.165em $\cdots$]{};
    \vertex (w3) at (75:1.1) [label=below:$w_t$]{};

    \vertex (z1) at (30:2) [label=right:$z_1$]{};
    \vertex (z3) at (60:2) [label=above:$z_s$]{};

    \vertex (y1) at (165:1) [label=below:$y_1$]{};
    %\node[style={circle,minimum size=5pt}] (w2) at (40:1) [label=center:\hskip0.165em $\cdots$]{};
    \vertex (y3) at (105:1.1) [label=below:$y_r$]{};

    \vertex (x1) at (150:2) [label=left:$x_1$]{};
    \vertex (x3) at (120:2) [label=above:$x_q$]{};

    \begin{scope}[yshift=-30]

    \end{scope}

    \draw[-, thick, color=lightgray] (u) -- (v);
    \draw[-, thick] (v) -- (u3);
    \draw[-, thick, shorten >=5] (u3) -- (u4);
    \draw[-, thick, shorten <=5] (u4) -- (u5);
    \draw[-, thick] (u5) -- (w);
    \draw[-, thick, color=lightgray] (w) -- (u);

    \draw[-, thick] (u) -- (x3);
    \draw[-, thick, dashed] (x3) -- (x1);
    \draw[-, thick] (x1) -- (v);
    \draw[-, thick] (v) -- (y1);
    \draw[-, thick, dashed] (y1) -- (y3);
    \draw[-, thick] (y3) -- (u);

    \draw[-, thick] (u) -- (z3);
    \draw[-, thick, dashed] (z3) -- (z1);
    \draw[-, thick] (z1) -- (w);
    \draw[-, thick] (w) -- (w1);
    \draw[-, thick, dashed] (w1) -- (w3);
    \draw[-, thick] (w3) -- (u);

    \draw (0,-2.5) node{(d)};

\end{tikzpicture}

\end{center}
\caption{The cycle $C$ defined in the proof of Theorem~\ref{thm:chordalize}.
The gray colored edges represent the newly edges on the hole $H^*$ in $G^*$ and $w$ in (c) and (d) turns out to be $u_1$.}
\label{fig:example}
\end{figure}
\end{proof}

%For a vertex subset $\CCC$ of a graph $G$, we say that $\CCC$ satisfies the \emph{non-consecutive property} (\emph{NC property} for short) in $G$ if, for distinct vertices $u, v \in \CCC$, any hole in $\HHH_G(u)$ and any hole in $\HHH_G(v)$ do not share consecutive edges.

%It is easy to check that no hole in $G$ contains two vertices in $\CCC$ satisfying the NC property.

\begin{Cor}\label{cor:chordalize}
Suppose that a graph $G$ has a hole cover $\CCC = \{u_1, \ldots, u_k\}$ satisfying the NC property and that $G_0 = G$ and, for $i=1,\ldots,k$,  $G_i$ is the graph obtained by locally chordalizing the holes in $\HHH(G_{i-1}, u_i)$ by $u_{i}$.
Then $G_k$ is chordal.
Moreover, the resulting chordal graph is independent of the order of $u_1, \ldots, u_k$ by which the local chordalizations are performed.
\end{Cor}
\begin{proof}
By induction on the size $k$ of a hole cover satisfying the NC property.
If $k=1$, then $G_k$ is chordal by Theorem~\ref{thm:chordalize}.
Suppose that the statement is true for any graph with a hole cover with size $k-1$ satisfying the NC property.
%Now take a graph with a hole cover $\CCC := \{u_1, \ldots, u_k\}$ with size $k$ satisfying (P1).
Now we locally chordalize the holes in $\HHH(G, u_1)$ by $u_1$ to obtain $G_1$.
By Theorem~\ref{thm:chordalize}, $\CCC \setminus \{u_1\}$ is a hole cover of $G_1$.
By ($\natural$), $\CCC \setminus \{u_1\}$ still satisfies the NC property in $G_1$.
Therefore, by the induction hypothesis, $G_k$ is chordal.

It is sufficient to show the uniqueness for the case $k=2$.
Let $G'$ and $G''$ be the graphs obtained by locally chordalizing the holes in $\HHH(G, u_2)$ by $u_2$ and the holes in $\HHH(G', u_1)$ by $u_1$, respectively.

Since $\CCC$ satisfies the NC property, no hole in $G$ contains two vertices in $\CCC$.
Therefore, by Theorem~\ref{thm:chordalize}, $\HHH(G, u_1) = \HHH(G', u_1)$ and $\HHH(G_1, u_2) = \HHH(G, u_2)$, which implies $G_2 = G''$.
%Now take an edge $e$ in $G_2$.
%If $e$ belongs to $G$, then it belongs to $H_2$ and we are done.
%Now suppose that $e$ is not in $G$.
%Suppose that $e$ was added to $G_1$
%Then one end of $e$ is $u_1$ and the other end is on a hole in $\HHH_G(u_1)$.
%
%is incident with $u_1$
\end{proof}

Let $G$ be a graph with a hole cover $\CCC$ satisfying the NC property.
Corollary~\ref{cor:chordalize} says that a chordal graph can be obtained by applying local chordalizations recursively by the vertices in $\CCC$ and the resulting chordal graph is the same no matter which order of the vertices is taken.
The uniqueness of the resulting chordal graph allows us to denote it by a notation, say $\widehat{G}({\CCC})$.
In the rest of this paper, we derive some noteworthy theorems by utilizing $\widehat{G}(\CCC)$ for graphs $G$ having hole covers $\CCC$ satisfying the NC property.

\begin{Lem}\label{lem:newedge}
Let $G$ be a graph with a hole cover $\CCC$ satisfying the NC property.
Suppose that vertices $u$ and $w$ in $\CCC$ are adjacent in $G$.
%If a hole covered by $u$ and a hole covered by $v$ share a vertex $z$, then $z$ is adjacent to either $u$ or $v$ in $G$.
Then any newly added edge incident to $u$ and any newly added edge incident to $w$ are not adjacent in $\widehat{G}(\CCC)$.
\end{Lem}
\begin{proof}
Suppose to the contrary that there exist a newly added edge incident to $u$ and a newly added edge incident to $w$ which are adjacent in $\widehat{G}(\CCC)$.
Let $uv$ and $wv$ be such edges for some $v \in V(G)$.
Then, by the definition of local chordalization, neither $uv$ nor $wv$ is an edge in $G$ and there exist $H_u \in \HHH(G, u)$ and $H_w \in \HHH(G, w)$ sharing the vertex $v$.
%Let $H_u = u u_1 \cdots u_p v u_{p+1} \cdots u_q u$ and $H_w = w w_1 \cdots w_r v w_{r+1} \cdots w_s v$ where $1 \le p < q$ and $1 \le r < s$.

To reach a contradiction, suppose that there exist an internal vertex on a $(u,v)$-section of $H_u$ and an internal vertex on the other $(u,v)$-section of $H_u$ each of which is adjacent to $w$.
Then, by Lemma~\ref{lem:hole} with $P = w$, there is a hole in $G$ not containing $u$ but containing two consecutive edges on $H_u$ incident to $v$, which contradicts the hypothesis that $\CCC$ satisfies the NC property.
Therefore there exists one of the $(u,v)$-sections of $H_u$ such that $w$ is not adjacent to any internal vertex on it.
Let $Q$ be such a section.
By symmetry, we may conclude that there exists one of the $(v,w)$-sections of $H_w$ such that $u$ is not adjacent to any internal vertex on it.
Let $R$ be such a section.

Let $W$ be the concatenation of $Q$ and $R$ at $v$.
Then $W$ is a $(u,w)$-walk in $G - uw$.
Now $W$ contains a $(u,w)$-path $S$ as an induced subgraph in $G-uw$.
By the previous argument, the vertex immediately following $u$ on $S$ cannot be on $R$ while the vertex immediately preceding $w$ on $S$ cannot be on $Q$.
Therefore we may conclude that the length of $S$ is at least three.
Thus $S$ and the edge $uw$ form a hole in $G$.
However, this hole contains both $u$ and $w$, which is impossible as $\CCC$ satisfies the NC property.
\end{proof}

\begin{Thm}\label{thm:newclique}
Let $G$ be a graph with a hole cover $\CCC$ satisfying the NC property.
Suppose that a vertex set $K$ forms a clique in $\widehat{G}(\CCC)$ but not in $G$.
Then there exists a vertex $u \in K \cap \CCC$ such that $K \setminus \{u\}$ is a clique in $G$.
\end{Thm}
\begin{proof}
Since $K$ is a clique in $\widehat{G}(\CCC)$ but is not a clique in $G$, $K \cap \CCC \neq \emptyset$.
Suppose that $K \cap \CCC$ is not a clique in $G$.
Then there exist two vertices $x$ and $y$ in $K \cap \CCC$ such that $xy \notin E(G)$.
This implies that there exists a hole in $G$ containing both $x$ and $y$, which is impossible by the hypothesis that $\CCC$ satisfies the NC property.
Therefore $K \cap \CCC$ is a clique in $G$.
However, $K$ is not a clique in $G$, so there exist vertices $u \in K \cap \CCC$ and $v \in K \setminus \CCC$ such that $uv$ is a newly added edge.
We claim that every newly added edge whose end vertices belong to $K$ is incident with $u$ by contradiction.
Suppose that there exists a newly added edge $zw$ such that $\{z, w\} \subset K \setminus \{u\}$.
By the definition of $\widehat{G}(\CCC)$, we may assume $z \in \CCC$ and $w \notin \CCC$.
Since $K \cap \CCC$ is a clique in $G$, $zu \in E(G)$.
Then Lemma~\ref{lem:newedge} implies that $v \neq w$, and $uw$ and $zv$ are edges in $G$.
If $vw$ is a newly added edge, then either $v$ or $w$ belongs to $\CCC$, which is not the case.
Therefore $vw \in E(G)$.
Then the cycle $uzvwu$ is obviously a hole in $G$ containing $u$ and $z$, which contradicts the hypothesis that $\CCC$ satisfies the NC property.
Thus we have shown that every newly added edge in $K$ is incident with $u$.
Hence $K \setminus \{u\}$ is a clique in $G$.
\end{proof}

\begin{Cor}\label{cor:newclique}
Let $G$ be a graph with a hole cover $\CCC$ satisfying the NC property.
Then $\omega(\widehat{G}(\CCC)) \le \omega(G) + 1$.
Furthermore the equality holds
if and only if
\begin{itemize}
\item[$(\dag)$] There exists a vertex $u \in \CCC$ such that the set $\left\{\left( \bigcup_{H \in \HHH(G, u)} V(H) \right) \cup N_G(u) \right\} \setminus \{u\}$ contains a maximum clique $K$ of $G$.
\end{itemize}
\end{Cor}
\begin{proof}
By Theorem~\ref{thm:newclique}, $\omega(\widehat{G}(\CCC)) \le \omega(G) + 1$.
%To show the ``furthermore part'', suppose that $\omega(\widehat{G}(\CCC)) = \omega(G) + 1$.
Furthermore, by the same theorem,  $\omega(\widehat{G}(\CCC)) = \omega(G) + 1$ if and only if there is a clique $K$ in $\widehat{G}(\CCC)$ of size $\omega(G) + 1$ and there is a vertex $u \in K \cap \CCC$  such that $K \setminus \{u\}$ forms a clique in $G$, which is equivalent to ($\dag$).
\end{proof}

%\begin{Conj}
%$\chi(G) \le \omega(G) + \lceil \frac{i(G)}{2} \rceil$
%\end{Conj}

\begin{Thm}\label{thm:knminor}
Let $G$ be a graph with a hole cover $\CCC$ satisfying the NC property.
Then every clique of $\widehat{G}(\CCC)$ is a minor of $G$.
%If $G$ is $K_n$-minor-free, then $\widehat{G}(\CCC)$ is $K_n$-free.
\end{Thm}
\begin{proof}
%Suppose to the contrary that $G$ is $K_n$-minor-free while $\widehat{G}(\CCC)$ contains a clique $K$ of size $n$.
Let $K$ be a clique in $\widehat{G}(\CCC)$ of size $n$.
If $K$ is a clique in $G$, then we are done.
Suppose that $K$ is not a clique in $G$.
By Theorem~\ref{thm:newclique}, there exists a vertex $u \in K \cap \CCC$ such that $K \setminus \{u\}$ is a clique in $G$.
Therefore $|V(H) \cap (K \setminus \{u\})| \le 2$ for every $H \in \HHH(G, u)$.
Furthermore, every newly added edge whose end vertices are in $K$ is incident with $u$.

Let $uv_1, \ldots, uv_l$ be the newly added edges whose end vertices are in $K$ and $X = \{v_1, \ldots, v_l\}$.
%(By the definition of $\widehat{G}(\CCC)$, each of those edges is a chord of a hole in $\HHH(G, u)$.)
Take a vertex $v_i \in X$.
Then there exists $H \in \HHH(G, u)$ containing $v_i$.
Since $uv_i$ is a newly added edge, $u$ and $v_i$ are not consecutive on $H$.
Then each of the $(u,v_i)$-sections of $H$ contains at least one internal vertex.
In addition, $V(H) \cap X \subset V(H) \cap (K \setminus \{u\})$.
Since we have shown that $|V(H) \cap (K \setminus \{u\})| \le 2$, $|V(H) \cap X| \le 2$.
Since $v_i  \in V(H) \cap X$, $V(H)$ contains at most one vertex in $X$ other than $v_i$.
Thus one of the $(u,v_i)$-sections of $H$ does not contain any vertex in $K$ as an internal vertex.
Let $P_i$ be such a section.
%{\bf (20170714)}
In $G$, we contract the edges on $P_1$ except the edge incident to $v_1$ to obtain the edge $e_1$ joining $u$ and $v_1$.
Then $P_2$ is transformed to a $(u,v_2)$-walk $W_2$ in the graph $G_1$ resulting from the contractions and still does not contain any vertex in $K$ other than $u$ and $v_2$ by the way of contractions and by the choice of $P_i$.
In $G_1$, we contract the edges on $W_2$ except the edge incident to $v_2$ to obtain the graph $G_2$ and the edge $e_2$ joining $u$ and $v_2$ in $G_2$.
We may repeat this process until we obtain the graph $G_l$ from $G_{l-1}$ and the edge $e_l$ joining $u$ and $v_l$ in $G_l$.
Now, $G_l$ contains the vertices of $K$ and the edges $uv_1, \ldots, uv_l$ so that $K$ is clique of size $n$ in $G_l$.
\end{proof}

\noindent
Now we have the following corollary.

\begin{Cor}\label{cor:knminor}
Let $G$ be a graph with a hole cover $\CCC$ satisfying the NC property.
If $G$ is $K_n$-minor-free, then $\widehat{G}(\CCC)$ is $K_n$-free.
\end{Cor}

Now we are ready to give a proof of Theorem~\ref{thm:coloring1}.

\bigskip
\emph{A proof of Theorem~\ref{thm:coloring1}}.
Since $\widehat{G}(C)$ is a chordal completion, $$\chi_{DP}(G) \le \chi_{DP}(\widehat{G}(C)) = \omega(\widehat{G}(C)).$$
By Corollary~\ref{cor:newclique}, $\omega(\widehat{G}(C)) \le \omega(G)+1$, so $\chi_{DP}(G) \le \omega(G)+1$.
Moreover, by Corollary~\ref{cor:knminor}, if $G$ is $K_n$-minor-free, then $\omega(\widehat{G}(C)) \le n-1$ and so $\chi_{DP}(G) \le n-1$.
\hfill {\large $\Box$}

\section{A partial result on the Erd\H{o}s-Faber-Lov\'{a}sz Conjecture}
%https://faculty.math.illinois.edu/~west/regs/efl.html
%https://faculty.math.illinois.edu/~west/regs/efl.html
%https://faculty.math.illinois.edu/~west/regs/efl.html

The following is one of the versions equivalent to the conjecture given by
Erd\H{o}s, Faber, and Lov\'{a}sz in 1972.

\begin{Conj}\label{conj:erdos}
%({\bf Modify the statement}) If $k$ complete graphs, each having exactly $k$ vertices, have the property that every pair of complete graphs has at most one shared vertex, then the union of the graphs is $k$-colorable.
If $G$ is the union of $k$ edge-disjoint copies of $K_k$ for a positive integer $k$, then $\chi(G)=k$.
\end{Conj}

\noindent
In this section, we show that the above conjecture is true for the graphs satisfying the NC property by deriving the following theorem.

\begin{Thm}\label{thm:erdos}
%Suppose that a graph $G$ satisfies the NC property and has an edge clique cover consisting of $k$ maximal cliques of size $k$ and satisfying the condition in Conjecture~\ref{conj:erdos}.
%Then $G$ is $k$-colorable.
If a graph $G$ satisfying the NC property is the union of $k$ edge-disjoint copies of $K_k$ for a positive integer $k$, then $\chi_{DP}(G)=k$.
\end{Thm}

We start by showing the following lemmas.
A vertex is said to be a \emph{simplicial vertex} if its neighbors form a clique.

\begin{Lem}\label{lem:erdos0}
Let $G$ be a graph and $L$ be a maximal clique of $G$.
Suppose that every vertex in $G-L$ is a simplicial vertex in $G$.
%any two vertices $u$ and $v$ in the same component of $G-L$ satisfy $N_G(u) \cap L = N_G(v) \cap L$.
Then $G$ is chordal.
\end{Lem}
\begin{proof}
It suffices to prove the lemma when $G$ is connected.
Suppose to the contrary that $G$ has a hole $H$.
Since $L$ is complete and $H$ is a hole in $G$, $|V(H) \cap L| \le 2$.
Then $V(H) \setminus L$ forms an induced path in $G$ and, by the hypothesis that any vertex in $G-L$ is a simplicial vertex in $G$, $|V(H) \setminus L| \le 2$.
Since $H$ is a hole, $4 \le |V(H)| = |V(H) \cap L| + |V(H) \setminus L| \le 4$ and so $|V(H) \cap L| = 2$ and $|V(H) \setminus L| = 2$.
Since $V(H) \setminus L := \{u,v\}$ and $V(H) \cap L := \{x,y\}$ are cliques in $G$, $uv$ and $xy$ are edges in $G$.
Since $H$ is a hole, $u$ cannot be a simplicial vertex in $G$ and we reach a contradiction.
\end{proof}

%\begin{Lem}\label{lem:erdos1}
%Let $G$ be the union of the graphs in $\LLL$ consisting of $k$ edge-disjoint copies of $K_k$.
%%Suppose that a graph $G$ has an edge clique cover $\LLL$ consisting of $k$ maximal cliques of size $k$ and satisfying the condition in Conjecture~\ref{conj:erdos}.
%If $L$ is a maximal clique of size $l$ not belonging to $\LLL$, then $l \le k$.
%Furthermore, if $l = k$, then $G$ is chordal.
%\end{Lem}
In a graph, we say that \emph{a clique $K$ covers an edge $e$} if $e$ is an edge of $K$.

\begin{Lem}\label{lem:erdos1}
Let $G$ be a union of $k$ edge-disjoint copies of $K_k$ and $\LLL$ be the set of those $k$ copies of $K_k$ for a positive integer $k$. Then $\omega(G) = k$.
Furthermore, if a maximal clique of $G$ with size $k$ does not belong to $\LLL$, then $G$ is chordal.
\end{Lem}
\begin{proof}
Since $G$ contains $K_k$, $\omega(G) \ge k$.
We prove that any maximal clique of $G$ not belonging to $\LLL$ has size at most $k$ to show $\omega(G) \le k$.
Let $L$ be a maximal clique of $G$ with size $l$ which does not belong to $\LLL$.
%For each vertex $u$ in $L$, let $E_L(u) = [u, L\setminus \{u\}]$ be the set of edges on $L$ incident to $u$.
For each vertex $u$ in $L$, let $n_u$ be the minimum number of cliques in $\LLL$ needed to cover the edges in the edge cut $[u, L\setminus \{u\}]$.
Since each edge of $G$ is covered by a unique maximal clique in $\LLL$, $n_u$ is the number of cliques in $\LLL$ which share an edge with $L$.
Since $L$ is a maximal clique of $G$ and does not belong to $\LLL$, the edges on $L$ are covered by at least two cliques in $\LLL$ and so $n_{u} \ge 2$ for each $u \in L$.
Now let $u^*$ be a vertex in $L$ with the minimum $p := n_{u^*}$.
By the observation that $n_{u} \ge 2$ for each $u \in L$, $p \ge 2$.
%Since the cliques in $\LLL$ are mutually edge-disjoint, each edge incident to $u^*$ which is covered by $L$ is covered by exactly one maximal clique in $\LLL$.
%Thus there are exactly $p$ cliques in $\LLL$ to cover the edges incident to $u^*$ covered by $L$.
%Without loss of generality, we may assume
Let $L_1, \ldots, L_p$ be the cliques in $\LLL$ which cover the edges in $[u^*, L\setminus \{u^*\}]$.
Let $l_i = |L \cap L_i| - 1$ for each $i=1, \ldots, p$.
%Take an edge incident to $u^*$ which is covered by $L$.
%Then there exists exactly one clique in $\LLL$ covering that edge.
%We let $m_1$ be the number of vertices except $u^*$ common to $L$ and that clique.
%We may repeat this process to obtain $m_2$, $\ldots$, $m_p$.
Without loss of generality, we may assume
\begin{equation}\label{eqn:erdos2}
l_1 \ge l_2 \ge \cdots \ge l_p \ge 1.
\end{equation}
Suppose that there exist distinct vertices $u_1$ and $u_2$ in $L \cap L_i$ for some $i \in \{1,\ldots,p\}$ such that an edge in $[u_1, L \setminus L_i]$ and an edge in $[u_2, L \setminus L_i]$ are covered by the same clique $K$ in $\LLL$.
Then $K \neq L_i$.
However, since $K$ is a clique, $u_1u_2$ is covered by $K$, a contradiction to the hypothesis.
Therefore
\begin{itemize}
%\item[($\star$)] for distinct vertices $u_1, u_2$ in $L \cap L_i$, an edge in $[u_1, L \setminus L_i]$ and an edge in $[u_2, L \setminus L_i]$ are covered by distinct cliques in $\LLL$ for any $i=1,\ldots,p$.
\item[($\star$)] two edges in $[L \cap L_i,L \setminus L_i]$ are covered by distinct cliques in $\LLL$ for $i=1,\ldots,p$ unless they have a common end in $L \cap L_i$.
\end{itemize}

Since $L_1, \ldots, L_p$ are mutually edge-disjoint,
\begin{align}\label{eqn:erdos1}
l &= \left|L \cap \bigcup_{i=1}^pL_i\right|
= \left|\left( \bigcup_{i=1}^p  (L \cap L_i) \setminus \{u^*\}\right) \cup \{u^*\}\right| \notag \\
&=\sum_{i=1}^p |(L \cap L_i) \setminus \{u^*\}|+1 = \sum_{i=1}^p l_i + 1.
\end{align}
Since $p \ge 2$, $L_1$ and $L_2$ exist.
Each edge in $[(L \cap L_1)\setminus\{u^*\}, (L \cap L_2)\setminus\{u^*\}]$ is covered by exactly one clique in $\LLL$ by the hypothesis.
Since any edge in $[(L \cap L_1)\setminus\{u^*\}, (L \cap L_2)\setminus\{u^*\}]$ is not incident to $u^*$, any clique in $\LLL$ covering an edge in $[(L \cap L_1)\setminus\{u^*\}, (L \cap L_2)\setminus\{u^*\}]$ cannot be $L_i$ for any $i=1,\ldots,p$.
Therefore we need at least $p+l_1l_2$ cliques in $\LLL$ to cover the edges in $[u^*,L\setminus\{u^*\}] \cup [(L \cap L_1)\setminus\{u^*\}, (L \cap L_2)\setminus\{u^*\}]$ and so
%Since each edge in $G$ is covered by exactly one clique in $\LLL$ by the hypothesis, the edges in $[u^*,L\setminus\{u^*\}] \cup [L_1\setminus\{u^*\}, L_2\setminus\{u^*\}]$ are covered by exactly $p + l_1l_2$ cliques in $\LLL$ and so
%Since each edge in $[L_1, L_2]$ is covered by exactly one clique in $\LLL$,
\begin{equation}\label{eqn:erdos3}
p + l_1l_2 \le |\LLL| = k.
\end{equation}
For each vertex $u$ in $L \cap L_1$, $n_u \ge p$ and so there are at least $p$ cliques in $\LLL$ needed to cover the edges in $[u,L \setminus \{u\}]$.
By ($\star$), we need at least $p + l_1(p-1)$ distinct cliques in $\LLL$ to cover the edges in $[u^*,L\setminus\{u^*\}] \cup [(L \cap L_1) \setminus \{u^*\}, L \setminus L_1]$ and so
%in $L$ incident to $u^*$ and those edges in $L$ each of which is incident to a vertex in $L \cap L_1$ and so
\begin{equation}\label{eqn:erdos4}
p + l_1(p-1) \le k.
\end{equation}

%{\it Case 1}.
%$l_2 \ge p$.

If $l_2 \ge p$, then
\begin{align*}
l &= \sum_{i=1}^p l_i + 1 \tag{by \eqref{eqn:erdos1}} \\
&\le l_1p + 1 \tag{by \eqref{eqn:erdos2}} \\
&< l_1l_2 + p \tag{by the case assumption  and the fact that $p \ge 2$}\\
&\le k \tag{by \eqref{eqn:erdos3}}.
\end{align*}
%where the first equality, the second and the fourth inequalities hold by \eqref{eqn:erdos1}, \eqref{eqn:erdos2}, \eqref{eqn:erdos3}, and the third inequality is true by the case assumption that $l_2 \ge p$ and the fact that $p \ge 2$.
Therefore we have shown that $l < k$ if $l_2 \ge p$ and so the ``furthermore'' part is vacuously true.
%{\it Case 2}.
%$l_2 \le p-1$.

Now assume $l_2 \le p-1$.
Then
\begin{align*}
l &= \sum_{i=1}^p l_i  + 1 \tag{by \eqref{eqn:erdos1}} \\
&\le (p-1)l_1 + l_2 + 1 \tag{by \eqref{eqn:erdos2}} \\
&\le (p-1)l_1 + p \tag{by the assumption that $l_2 \le p-1$} \\
&\le k \tag{by \eqref{eqn:erdos4}}
\end{align*}
To show the ``furthermore'' part, suppose $l=k$.
Then each of the three inequalities above becomes the equality.
Now, if $p=2$, then $l_2 = 1$ and $l = l_1 + l_2 + 1 = l_1 + 2 = k$, which implies $l_1 = k-2$.
If $p \ge 3$, then, by \eqref{eqn:erdos2}, $l_1 = \cdots = l_p = p-1$ and $k=p^2-p+1$.

\noindent
{\it Case 1}. $p=2$.
Let $L \cap L_1 = \{u^*, u_1, u_2, \ldots, u_{k-2}\}$ and $L \cap L_2 = \{u^*, v\}$.
%Since $L$ is a clique and
Since $u_i$ and $v$ belong to $L$, $u_iv$ is an edge in $G$ for each $i=1,\ldots,k-2$.
Since $\LLL$ is an edge clique cover of $G$, there is a clique in $\LLL$ covering the edge $u_iv$ for each $i=1, \ldots, k-2$.
By ($\star$), no clique in $\LLL$ contains $u_i, u_j, v$ for $1 \le i < j \le k-2$.
Therefore, by relabelling the cliques in $\LLL$ if necessary, we may assume $L_{i+2}$ is a clique covering $u_iv$ for each $i=1, \ldots, k-2$.
Then $(L_1 \cap L_2) \cap L = \{u^*\}$, $(L_1 \cap L_i) \cap L = \{u_{i-2}\}$ for $i=3,\ldots,k$, and $(L_i \cap L_j) \cap L = \{v\}$ for $2 \le i < j \le k$.
Therefore $L_i$ and $L_j$ share exactly one vertex in $L$ for distinct $i, j$ in $\{1,\ldots,k\}$

%Suppose to the contrary that that $G-L$ has a vertex $x$ which is a non-simplicial vertex in $G$.
%Then $x$ belongs to $L_i$ and $L_j$ in $\LLL$ for some distinct $i$ and $j$.
%We have shown that $L_i$ and $L_j$ share a vertex $y$ in $L$.
%Since $x \notin L$, $x$ and $y$ are distinct and the edge $xy$ is covered by $L_i$ and $L_j$, which contradicts the hypothesis that the cliques in $\LLL$ is pairwise edge-disjoint.
%Therefore $x$ is a simplicial vertex.
%Since $x$ is arbitrarily chosen, every vertex in $G-L$ is a simplicial vertex in $G$.
%Thus, by Lemma~\ref{lem:erdos0}, $G$ is chordal.

\noindent
{\it Case 2}. $p \ge 3$.
Then $l_1 = \cdots = l_p = p-1$ and $k=p^2-p+1$.
Let $L \cap L_1 = \{u^*, v_1, \ldots, v_{p-1}\}$ and $L \cap L_2 = \{u^*, w_1, \ldots, w_{p-1}\}$.
%Fix $i$ and $j$ in $\{1,\ldots,p-1\}$.
Since $L$ is a clique in $G$, $v_i$ and $w_j$ are adjacent in $G$
%$v_i$ is adjacent to each vertex in $L \setminus L_1$and so each edge in $[v_i, L \setminus L_1]$
and the edge $v_iw_j$ must be covered by a clique in the edge clique cover $\LLL$ for any $i,j \in \{1,\ldots,p-1\}$.
Let $K_{i,j} \in \LLL$ be a clique which covers the edge $v_iw_j$ for $i,j \in \{1,\ldots,p-1\}$ and let $\KKK = \{K_{i,j} \mid i,j \in \{ 1, \ldots, p-1\} \}$.
%Take a clique $K \in \LLL$ which covers an edge $v_iw$ in $[v_i, L \setminus L_1]$.
Suppose $K_{i,j} = L_t$ for some $i,j \in \{1,\ldots,p-1\}$ and  $t \in \{1,\ldots,p\}$.
Then the edges $u^*w_j \in [L \cap L_1, L \setminus L_1]$ and $v_iw_j \in [L \cap L_1, L \setminus L_1]$ are covered by $K_{i,j}$, which is impossible by ($\star$).
Therefore $K_{i,j}$ cannot be any of $L_1, \ldots, L_p$.
By ($\star$), $K_{i,j} \neq K_{i',j'}$ if $(i,j) \neq (i',j')$.
Therefore $|\KKK| = (p-1)^2$ and
\[
|\{L_1, \ldots, L_p\} \cup \KKK| = p + (p-1)^2 = p^2 - p + 1.
\]
Since $|\LLL| = k$  $ = p^2 - p + 1$, $\LLL = \{L_1, \ldots, L_p\} \cup \KKK$.

To apply Lemma~\ref{lem:erdos0}, we first claim that $M \cap N \subset L$ for any distinct cliques $M, N \in \LLL$.
Take two distinct cliques $M$ and $N$ in $\LLL$.
If $M$ and $N$ belong to $\{L_1,\ldots, L_p\}$, then $M \cap N = \{u^*\} \subset L$.
Suppose that one of $M$ and $N$ is in $\{L_1,\ldots, L_p\}$ and the other is in $\KKK$.
Without loss of generality, we may assume $M = L_t := \{u^*, x_1, \ldots, x_{p-1}\}$ and $N = K_{i,j}$ for some $t \in \{1,\ldots,p\}$ and $i,j \in \{1,\ldots,p-1\}$.
By the hypothesis that the cliques in $\LLL$ are mutually edge-disjoint, %By the choice of $K_{i,j}$,
\[
L_t \cap K_{i,j} =
\begin{cases}
  \{v_i\} & \mbox{if $t=1$} \\
  \{w_j\} & \mbox{if $t=2$}.
\end{cases}
\]
Therefore $M \cap N = L_t \cap K_{i,j} \subset L$ for $t=1,2$.
Assume $3 \le t \le p$.
%(Note that $u^*v_i$ is covered by $L_1$ and $u^*x_r$ is covered by $L_t$ for any $r=1,\ldots,p-1$.
%In addition, $[\{v_1, \ldots, v_{p-1}\}, \{ x_1, \ldots, x_{p-1}\}] \subset [L \cap L_1, L \setminus L_1] \cap [L \cap L_t, L \setminus L_t]$.)
%(By ($\star$), no edge in $[\{v_1, \ldots, v_{p-1}\}, \{ x_1, \ldots, x_{p-1}\}]$ is covered by any of $L_1, \ldots, L_p$.)
Note that
\begin{equation}\label{eqn:star1}
E_{1t} := [\{v_1, \ldots, v_{p-1}\}, \{ x_1, \ldots, x_{p-1}\}] \subset [L \cap L_1, L \setminus L_1] \cap [L \cap L_t, L \setminus L_t].
\end{equation}
Suppose that an edge $v_rx_s$ is covered by $L_a$ for some $a \in \{1,\ldots,p\}$.
Then the edges $u^*x_s$ and $u^*v_r$ are covered by $L_a$.
However, $u^*$ and $v_r$ belong to $L \cap L_1$,
$\{u^*x_s, v_rx_s\} \in [L \cap L_1, L \setminus L_1]$, and we reach a contradiction to ($\star$).
Therefore each edge in $E_{1t}$ should be covered by a clique in $\KKK$.
Since $\KKK \subset \LLL$, it follows from ($\star$) that each clique in $\KKK$ covers at most one edge in $E_{1t} \subset [L \cap L_1, L \setminus L_1] \cap [L \cap L_t, L \setminus L_t]$.
Since $|\KKK| = (p-1)^2 = |E_{1t}|$, each clique in $\KKK$ covers exactly one edge in $E_{1t}$.
Therefore $K_{i,j}$ covers $v_rx_s$ for some $r,s \in \{1,\ldots,p-1\}$.
Thus $L_t \cap K_{i,j}$ contains the vertex $x_s$.
By the hypothesis that the cliques in $\LLL$ are mutually edge-disjoint, $L_t \cap K_{i,j} = \{x_s\} \subset L$.
Hence $M \cap N \subset L$ for $M = L_t$ and $N = K_{i,j}$.
Finally we suppose that $M$ and $N$ belong to $\KKK$.
Then $M = K_{i,j}$ and $N = K_{i', j'}$ for some $i,i',j,j' \in \{1, \ldots, p-1\}$ with $(i,j) \neq (i',j')$.
If $i=i'$, then $M \cap N= \{v_i\} \subset L$ by the hypothesis.
Suppose $i \neq i'$.
Take a vertex $y \in L \setminus L_1$.
Since $L$ is a clique and $\{v_i, v_{i'}, y\} \subset L$,
$v_iy$ and $v_{i'}y$ are edges of $G$ and should be covered by cliques in $\LLL$.
%({\bf 20180112})
%By ($\star$), they should be covered by distinct cliques in $\LLL$.
We note that $L_b$ covers $u^*y$ if $L_b$ covers $v_iy$ or $v_{i'}y$ for any $b \in \{1,\ldots,p\}$.
%which contains $v_iy$ or $v_{i'}y$ whenever $v_iy
Therefore, by the hypothesis that the cliques in $\LLL$ are mutually edge-disjoint, $v_iy$ and $v_{i'}y$ are covered by cliques in $\KKK$. % by ($\star$).
Let $K_{c,d}$ be a clique in $\KKK$ covering $v_iy$.
Then $v_c, v_i, y$ belong to $K_{c,d}$.
Since $K_{c,d}$ is a clique, $v_c$ and $y$ are adjacent.
Then $v_cy$ and $v_iy$ belong to $[L \cap L_1, L \setminus L_1]$ and are covered by $K_{c,d}$.
Thus, by ($\star$), $v_i=v_c$ and so $i=c$.
Similarly, $v_{i'}y$ is covered by $K_{i',d'}$ for some $d' \in \{1,\ldots,p-1\}$.
%$K_{i,s}$ and $K_{i',s'}$, respectively, for some $s,s' \in \{1,\ldots,p-1\}$.
%By ($\star$), they are $K_{i,s}$ and $K_{i',s'}$, respectively, for some $s,s' \in \{1,\ldots,p-1\}$.
By the hypothesis on $\LLL$, $K_{i,d}$ and $K_{i',d'}$ are the unique cliques in $\LLL$ covering $v_iy$ and $v_{i'}y$, respectively.
As $K_{i,d}$ and $K_{i',d'}$ are uniquely determined by $y$, we may denote $K_{i,d}$ and $K_{i',d'}$ by $A(y)$ and $B(y)$, respectively.
Now we define a function $F : L \setminus L_1 \to \{(K_{i,q}, K_{i',q'}) \mid 1 \le q, q' \le p-1\}$ by $F(y) = (A(y), B(y))$ for $y \in L \setminus L_1$.
Then $F$ is well-defined.
By the hypothesis on $\LLL$ again, $A(y) \cap B(y) = \{y\}$ for each $y \in L \setminus L_1$ and so $F$ is injective.
Since the domain and the codomain of $F$ have the same cardinality $(p-1)^2$, $F$ is bijective.
Since $M$ and $N$ belong to $\KKK$, $(M,N)$ is contained in the codomain of $F$ and so there exists a vertex $z \in L \setminus L_1$ such that $F(z) = (M,N)$.
Then $M = A(z)$ and $N = B(z)$, so $M \cap N = A(z) \cap B(z) = \{z\} \subset L$.
Hence we have shown that $M \cap N \subset L$ for any distinct cliques $M$ and $N$ in $\LLL$.

In both cases, we have shown that $M \cap N \subset L$ for any distinct cliques $M$ and $N$ in $\LLL$.
Now we will show that every vertex in $G - L$ is simplicial in $G$.
Take a vertex $v$ in $G - L$.
Suppose to the contrary that $v$ is not a simplicial vertex in $G$.
Then $v$ has two neighbors $z_1$ and $z_2$ which are nonadjacent in $G$.
Since $\LLL$ is an edge clique cover of $G$, $\LLL$ contains a clique covering $vz_1$ and a clique covering $vz_2$.
Since $z_1$ and $z_2$ are nonadjacent, these two cliques are distinct.
However, they share a vertex $v$ which is not in $L$.
This contradicts our claim that the intersection of any two cliques in $\LLL$ is a subset of $L$.
Therefore every vertex in $G-L$ is a simplicial vertex in $G$.
Thus, by Lemma~\ref{lem:erdos0}, $G$ is chordal.
\end{proof}

%The following theorem tells us that the Erd\H{o}s, Faber, and Lov\'{a}sz conjecture is true for the graphs satisfying the NC property.

\emph{A proof of Theorem~\ref{thm:erdos}.}
Let $G$ be a graph satisfying the NC property which is  the union of $k$ edge-disjoint copies $L_1, \ldots, L_k$ of $K_k$.
Obviously $\chi_{DP}(G) \ge k$.
By Lemma~\ref{lem:erdos1}, $\omega(G) = k$.
Let $\LLL = \{L_1, \ldots, L_k\}$.
Then $\LLL$ is an edge clique cover consisting of cliques of size $k$.

Fix $i \in \{1,\ldots,k\}$.
Then $|L_i \cap L_j| \le 1$ for any $j \in \{1,\ldots,k\} \setminus \{i\}$.
Since $L_i$ has $k$ vertices, $L_i$ has a vertex $v$ not contained in $L_j$ for any $j \in \{1,\ldots,k\} \setminus \{i\}$.
Then $v$ is a simplicial vertex of $G$.
Since $i$ is arbitrarily chosen, $L_i$ has a simplicial vertex for any $i=1,\ldots,k$.
%\begin{itemize}
%\item[($\S$)] $L_i$ has a simplicial vertex for any $i=1,\ldots,k$.
%\end{itemize}

If $G$ is chordal, then $\chi_{DP}(G) = \omega(G) = k$ by ($\S$).
Now we suppose that $G$ is non-chordal.
Then, by the ``furthermore part'' of Lemma~\ref{lem:erdos1}, any clique not belonging to $\LLL$ has size less than $k$.
Since $L_i$ has a simplicial vertex of $G$, we may take a simplicial vertex from $L_i$ and denote it by $v_i$ for each $i=1,\ldots,k$.
Let $G' = G - \{v_1, \ldots, v_k\}$.
Then $G'$ still satisfies the NC property.
Since any clique not belonging to $\LLL$ has size less than $k$, $\omega(G') = k-1$.
Let $\CCC$ be a hole cover of $G'$ satisfying the NC property.
Then $\widehat{G'}(\CCC)$ is chordal by definition and, by Corollary~\ref{cor:newclique}, $\omega(\widehat{G'}(\CCC)) \le \omega(G') + 1 = k$.
Let $G^*$ be the graph obtained from $\widehat{G'}(\CCC)$ by adding the vertices $v_1, \ldots, v_k$ and the edges which were incident to $v_1, \ldots, v_k$ in $G$.
Then $G$ is a spanning subgraph of $G^*$.
Since $v_1, \ldots, v_k$ are simplicial vertices of $G$, they are still simplicial vertices of $G^*$.
Therefore, the fact that $\widehat{G'}(\CCC)$ is chordal implies that $G^*$ is chordal.
Moreover, we note that exactly $k-1$ edges are added for $v_i$ for each $i=1,\ldots,k$ to obtain $G^*$ from $\widehat{G'}(\CCC)$.
Then, since $\omega(\widehat{G'}(\CCC)) \le k$,
\[
k \le \chi_{DP}(G) \le \chi_{DP}(G^*) = \omega(G^*) \le k
\]
and so $\chi_{DP}(G) = k$.
\hfill {\large $\Box$}

\section{A minimal chordal completion of a graph}

\subsection{Non-chordality indices of graphs}

Given a graph $G$, we apply a sequence of local chordalizations to obtain a chordal completion $G^*$ of $G$ as follows:
Let $\CCC =\{v_1, \ldots, v_l\}$ be a hole cover of $G$ and $G_0 = G_0^* = G - \CCC$.
By the definition of hole cover, $G_0^*$ is chordal.
Let $G_1$ be the graph with
\[
V(G_1) = V(G_0^*) \cup \{v_1\} \quad \text{and} \quad  E(G_1) = E(G_0^*) \cup E \left( G - \bigcup_{j=2}^l \{v_j\}\right).
\]
Obviously $\{v_1\}$ is a hole cover of $G_1$ satisfying the NC property.
By Corollary~\ref{cor:chordalize}, we obtain the chordal graph $G_1^* = \widehat{G_1}(\{v_1\})$.
Let $G_2$ be the graph with
\[
V(G_2) = V(G_1^*) \cup \{v_2\} \quad \text{and} \quad  E(G_2) = E(G_1^*) \cup E \left( G - \bigcup_{j=3}^l \{v_j\}\right).
\]
Again, $\{v_2\}$ is a hole cover of $G_2$ satisfying the NC property.
Let $G_2^* = \widehat{G_2}(\{v_2\})$ and we repeat this process until we obtain the chordal graph $G_l^* = \widehat{G_{l}}(\{v_{l}\})$ as a desired graph $G^*$.
Then $G_l^*$ is a chordal completion of $G$.
We note that if $G$ is chordal, then $G=G_l^*$.
Now we have shown the following theorem.

In the rest of this paper, for the notation $\bigcup_{j=p}^q S_j$ of a finite union of sets, we assume that it refers to an empty set if $p>q$.

\begin{Thm}\label{thm:existence}
Let $G$ be a graph with a hole cover $\CCC$.
Then $\CCC$ can be partitioned into $\CCC_1, \ldots, \CCC_k$ for some positive integer $k$ so that
\begin{itemize}
  \item[(i)] $\CCC_i$ is a hole cover of the graph $G_{i}$ satisfying the NC property,
  \item[(ii)] $G_{i}^*$ is chordal,
\end{itemize}
where $G_0 = G_0^* = G - \CCC$; $G_{i}$ is the graph defined by $V(G_i) = V(G_{i-1}^*) \cup \CCC_{i}$,
      \[
         E(G_i) =
         %E(G_{i-1}^*) \cup \{vw \in E(G) \mid v \in \CCC_i, w \in V(G_{i-1}^*)\} ( =
         E(G_{i-1}^*) \cup E \left( G - \bigcup_{j=i+1}^k \CCC_j\right),
   \]
and $G_i^* = \widehat{G_i}(\CCC_i)$ for each $i=1,\ldots,k$.
\end{Thm}
\noindent
%We note that, in Theorem~\ref{thm:existence}, $G_i$ is an induced subgraph of $G_j$ for any $j \ge i \ge 0$.
Let $G$ be a graph with a hole cover $\CCC$.
We call an ordered partition $(\CCC_1, \ldots, \CCC_k)$ of a hole cover $\CCC$ satisfying the
conditions (i) and (ii) in Theorem~\ref{thm:existence} a \emph{local chordalization partition} of $\CCC$.
Then the graphs $G_i, G_i^*$ are uniquely determined by the given local chordalization partition $\tilde{\CCC} := (\CCC_1, \ldots, \CCC_k)$ of $\CCC$.
We call the process of obtaining $G_i$ and $G_i^*$ the \emph{chordalization chain} corresponding to $\tilde{\CCC}$.
Especially, we write the process of obtaining $G_i$ from $G_{i-1}^*$ as $G_{i-1}^* <_{\CCC_i} G_i$ (in the context that $G_{i-1}^*$ is a proper subgraph of $G_i$, we use ``strictly less'' notation) for $i=1, \ldots, k$.
Then the chordalization chain corresponding to $\tilde{\CCC}$ may be represented as
\[
G_0 = G_0^* <_{\CCC_1} G_1 \le G_1^* <_{\CCC_2}  G_2 \le G_2^* < \cdots <_{\CCC_k} G_k \le G_k^*.
\]
We note that $G_k^*$ is a chordal completion of $G$.
%Given a hole cover $\CCC$ of a graph $G$ and a local chordalization $\CCC$,
%Lemma~\ref{lem:existence} tells us that, given a non-chordal graph $G$, we may apply a sequence of local chordalizations to obtain a chordal graph $G^*$ so that $G$ is a spanning subgraph of $G^*$.
By the way, the last chordal completion in the chordalization chain corresponding to $\tilde{\CCC}$ is a minimal chordal spanning supergraph of $G$.

\begin{Prop}\label{prop:minimal}
Let $G$ be a graph,
$\tilde{\CCC} = (\CCC_1, \ldots, \CCC_\ell )$ be a local chordalization partition of a hole cover $\CCC$ of $G$,
and $G^*$ be the last graph in the chordalization chain corresponding to $\tilde{\CCC}$.
Then $G^*$ is a minimal chordal completion of $G$.
\end{Prop}
\begin{proof}
Let $H$ be a graph that is a spanning supergraph of $G$ and a proper subgraph of $G^*$.
Then $E(G^*) \setminus E(H) \neq \emptyset$.
By definition, each edge of $E(G^*) \setminus E(H)$ is incident to one of vertices in $\CCC$.
Let $s$ be the smallest index such that some vertices in $\CCC_s$ are incident to edges in  $E(G^*) \setminus E(H)$.
Now let $B$ be the set of edges in $E(G^*) \setminus E(H)$ which are incident to vertices in $\CCC_s$.
By the definition of local chordalization, $G_s^* - B$ is not chordal.
Thus there exists a hole  $C$ in $G_s^* - B$.
By the choice of $s$, the edges in $E(G^*)\setminus E(G_s^*)$ are incident to vertices in $\bigcup_{j=s+1}^{\ell} \CCC_{j}$.
By definition, $(\bigcup_{j=s+1}^{\ell} \CCC_{j}) \cap V(G_s^*) =\emptyset$.
Since $V(G_s^*)=V(G_s^*-B)$, the edges in $E(G^*) \setminus E(G_s^*)$ cannot be chords of $C$.
Since $E(H) \subset E(G^*)$, the edges in $E(H) \setminus E(G_s^*)$ cannot be chords of $C$.
Therefore $C$ is a hole in $H$ and so $H$ is not  chordal.
Hence we have shown that $G^*$ is a minimal chordal completion of $G$.
%
%Now any hole in $G_s^* - B$ is a hole in $H$ since the edges in $E(G^*)\setminus (E(H) \cup B)$ are incident to vertices in $\bigcup_{j=s+1}^{\ell} \CCC_{j}$.
%Hence we have shown that $G^*$ is a minimal chordal completion of $G$.
\end{proof}

Now we are ready to introduce a parameter of a graph which measures the number of steps of adding new edges to reach one of its chordal completion.

\begin{Defi}\label{def:non-chordality index}
The \emph{non-chordality index} of a graph $G$, denoted by ${i}(G)$, is defined as follows:
If $G$ is chordal, $i(G) = 0$.
If $G$ is not chordal, then $i(G)$ is defined to be the smallest $k$ over all the hole covers of $G$ in Theorem~\ref{thm:existence}.
\end{Defi}

\begin{Rem}\label{rmk:ncproperty}
A graph $G$ satisfies the NC property if and only if $G$ satisfies $i(G) \le 1$.
\end{Rem}

\begin{Ex}
We consider the graph $G$ given in Figure~\ref{fig:noneg}.
Since $G$ does not satisfy the NC property, $i(G) \ge 2$ by Remark~\ref{rmk:ncproperty}.
It is easy to check that $\CCC = \{u, v, w, x\}$ is a hole cover of $G$.
See Figure~\ref{fig:index} for an illustration.
Since $G_2^*$ is a chordal completion of $G$, $i(G) \le 2$.
Thus $i(G) = 2$.

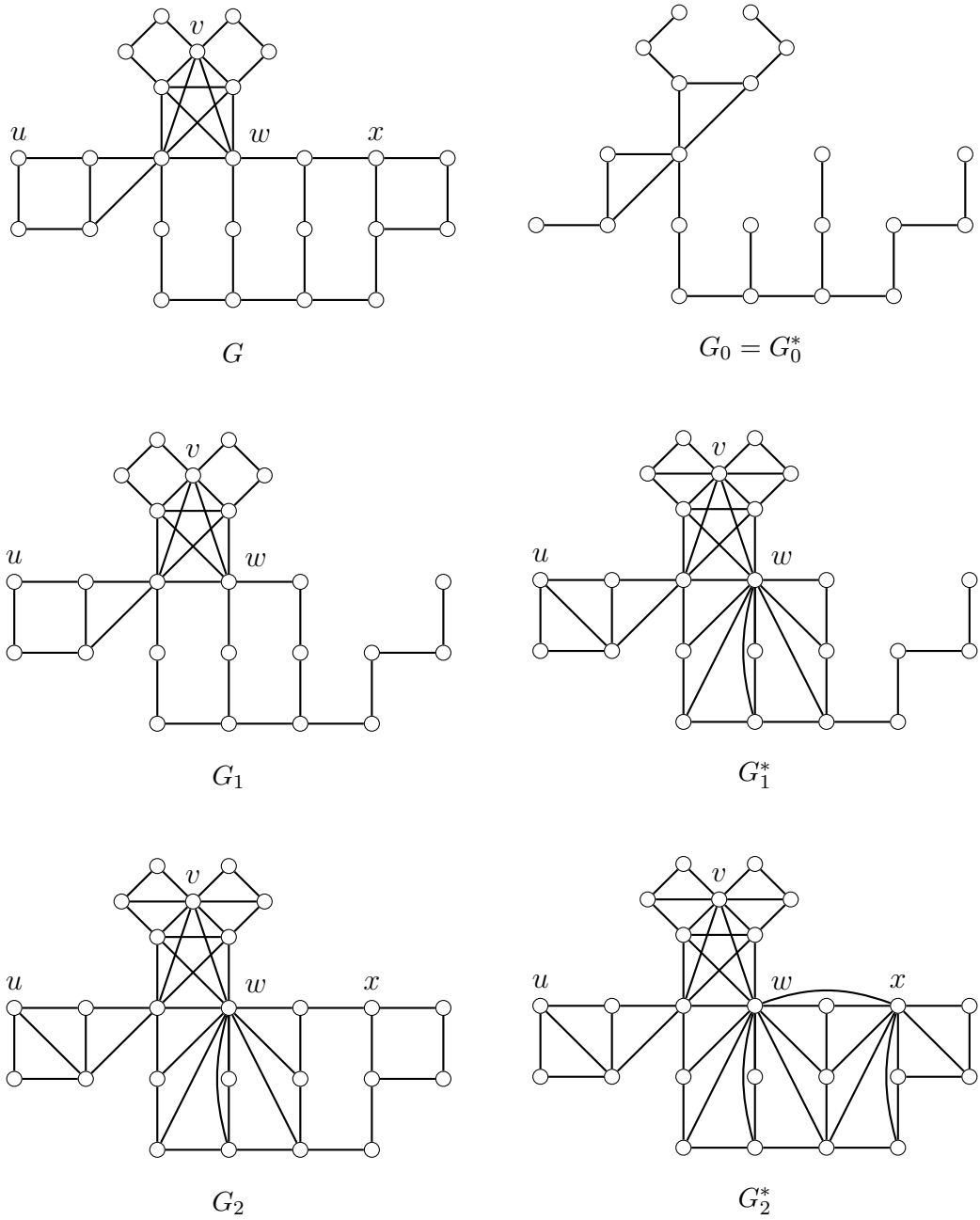
\begin{figure}
\begin{center}

\begin{tikzpicture}[x=1cm, y=1cm]

    \draw (3,-2.75) node{\small $G$};

    \vertex (a1) at (0,0) [label=above:$u$]{};
    \vertex (a2) at (0,-1) [label=above:$$]{};
    \vertex (a3) at (1,-1) [label=above:$$]{};
    \vertex (a4) at (1,0) [label=above:$$]{};

    \vertex (b1) at (2,0) [label=above:$$]{};
    \vertex (b2) at (2,-1) [label=above:$$]{};
    \vertex (b3) at (2,-2) [label=above:$$]{};

    \vertex (b4) at (3,0) [label=above right:$w$]{};
    \vertex (b5) at (3,-1) [label=right:$$]{};
    \vertex (b6) at (3,-2) [label=right :$$]{};

    \vertex (b7) at (4,0) [label=above:$$]{};
    \vertex (b8) at (4,-1) [label=above:$$]{};
    \vertex (b9) at (4,-2) [label=above:$$]{};

    \vertex (b10) at (5,0) [label=above:$x$]{};
    \vertex (b11) at (5,-1) [label=above:$$]{};
    \vertex (b12) at (5,-2) [label=above:$$]{};

    \vertex (b13) at (6,0) [label=above:$$]{};
    \vertex (b14) at (6,-1) [label=above:$$]{};

    \vertex (c1) at (2,1) [label=above:$$]{};
    \vertex (c2) at (1.5,1.5) [label=above:$$]{};
    \vertex (c3) at (2,2) [label=above:$$]{};
    \vertex (c4) at (2.5,1.5) [label=above:$v$]{};
    \vertex (c5) at (3,2) [label=above:$$]{};
    \vertex (c6) at (3.5,1.5) [label=above:$$]{};
    \vertex (c7) at (3,1) [label=above:$$]{};

    \path
    (a1) edge [-,thick] (a2)
    (a2) edge [-,thick] (a3)
    (a3) edge [-,thick] (a4)
    (a4) edge [-,thick] (a1)

    (b1) edge [-,thick] (a3)
    (b1) edge [-,thick] (a4)

    (b1) edge [-,thick] (b2)
    (b2) edge [-,thick] (b3)

    (b4) edge [-,thick] (b5)
    (b5) edge [-,thick] (b6)

    (b7) edge [-,thick] (b8)
    (b8) edge [-,thick] (b9)

    (b10) edge [-,thick] (b11)
    (b11) edge [-,thick] (b12)

    (b13) edge [-,thick] (b14)

    (b1) edge [-,thick] (b4)
    (b4) edge [-,thick] (b7)
    (b7) edge [-,thick] (b10)
    (b10) edge [-,thick] (b13)

%    (b2) edge [-,thick] (b5)
%    (b5) edge [-,thick] (b8)
%    (b8) edge [-,thick] (b11)
    (b11) edge [-,thick] (b14)

    (b3) edge [-,thick] (b6)
    (b6) edge [-,thick] (b9)
    (b9) edge [-,thick] (b12)

    (c1) edge [-,thick] (b1)
    (c1) edge [-,thick] (b4)
    (c4) edge [-,thick] (b1)
    (c4) edge [-,thick] (b4)
    (c7) edge [-,thick] (b1)
    (c7) edge [-,thick] (b4)

    (c1) edge [-,thick] (c2)
    (c2) edge [-,thick] (c3)
    (c3) edge [-,thick] (c4)
    (c4) edge [-,thick] (c1)
    (c4) edge [-,thick] (c5)
    (c5) edge [-,thick] (c6)
    (c6) edge [-,thick] (c7)
    (c7) edge [-,thick] (c1)
    (c7) edge [-,thick] (c4)

	;

\end{tikzpicture}
\qquad
\begin{tikzpicture}[x=1cm, y=1cm]

    \draw (3,-2.75) node{\small $G_0=G_0^*$};

%    \vertex (a1) at (0,0) [label=above:$u$]{};
    \vertex (a2) at (0,-1) [label=above:$$]{};
    \vertex (a3) at (1,-1) [label=above:$$]{};
    \vertex (a4) at (1,0) [label=above:$$]{};

    \vertex (b1) at (2,0) [label=above:$$]{};
    \vertex (b2) at (2,-1) [label=above:$$]{};
    \vertex (b3) at (2,-2) [label=above:$$]{};

%    \vertex (b4) at (3,0) [label=above right:$w$]{};
    \vertex (b5) at (3,-1) [label=above:$$]{};
    \vertex (b6) at (3,-2) [label=above :$$]{};

    \vertex (b7) at (4,0) [label=above:$$]{};
    \vertex (b8) at (4,-1) [label=above:$$]{};
    \vertex (b9) at (4,-2) [label=above:$$]{};

%    \vertex (b10) at (5,0) [label=above:$x$]{};
    \vertex (b11) at (5,-1) [label=above:$$]{};
    \vertex (b12) at (5,-2) [label=above:$$]{};

    \vertex (b13) at (6,0) [label=above:$$]{};
    \vertex (b14) at (6,-1) [label=above:$$]{};

    \vertex (c1) at (2,1) [label=above:$$]{};
    \vertex (c2) at (1.5,1.5) [label=above:$$]{};
    \vertex (c3) at (2,2) [label=above:$$]{};
%    \vertex (c4) at (2.5,1.5) [label=above:$v$]{};
    \vertex (c5) at (3,2) [label=above:$$]{};
    \vertex (c6) at (3.5,1.5) [label=above:$$]{};
    \vertex (c7) at (3,1) [label=above:$$]{};

    \path
    (a2) edge [-,thick] (a3)
    (a3) edge [-,thick] (a4)

    (b1) edge [-,thick] (a3)
    (b1) edge [-,thick] (a4)

    (b1) edge [-,thick] (b2)
    (b2) edge [-,thick] (b3)

    %(b4) edge [-,thick] (b5)
    (b5) edge [-,thick] (b6)

    (b7) edge [-,thick] (b8)
    (b8) edge [-,thick] (b9)

    %(b10) edge [-,thick] (b11)
    (b11) edge [-,thick] (b12)

    (b13) edge [-,thick] (b14)

    %(b1) edge [-,thick] (b4)
    %(b4) edge [-,thick] (b7)
    %(b7) edge [-,thick] (b10)
    %(b10) edge [-,thick] (b13)

%    (b2) edge [-,thick] (b5)
%    (b5) edge [-,thick] (b8)
%    (b8) edge [-,thick] (b11)
    (b11) edge [-,thick] (b14)

    (b3) edge [-,thick] (b6)
    (b6) edge [-,thick] (b9)
    (b9) edge [-,thick] (b12)

    (c1) edge [-,thick] (b1)
    %(c1) edge [-,thick] (b4)
    %(c4) edge [-,thick] (b1)
    %(c4) edge [-,thick] (b4)
    (c7) edge [-,thick] (b1)
    %(c7) edge [-,thick] (b4)

    (c1) edge [-,thick] (c2)
    (c2) edge [-,thick] (c3)
    %(c3) edge [-,thick] (c4)
    %(c4) edge [-,thick] (c1)
    %(c4) edge [-,thick] (c5)
    (c5) edge [-,thick] (c6)
    (c6) edge [-,thick] (c7)
    (c7) edge [-,thick] (c1)
    %(c7) edge [-,thick] (c4)

	;

\end{tikzpicture}

\vskip0.5cm

\begin{tikzpicture}[x=1cm, y=1cm]

    \draw (3,-2.75) node{\small $G_1$};

    \vertex (a1) at (0,0) [label=above:$u$]{};
    \vertex (a2) at (0,-1) [label=above:$$]{};
    \vertex (a3) at (1,-1) [label=above:$$]{};
    \vertex (a4) at (1,0) [label=above:$$]{};

    \vertex (b1) at (2,0) [label=above:$$]{};
    \vertex (b2) at (2,-1) [label=above:$$]{};
    \vertex (b3) at (2,-2) [label=above:$$]{};

    \vertex (b4) at (3,0) [label=above right:$w$]{};
    \vertex (b5) at (3,-1) [label=above:$$]{};
    \vertex (b6) at (3,-2) [label=above :$$]{};

    \vertex (b7) at (4,0) [label=above:$$]{};
    \vertex (b8) at (4,-1) [label=above:$$]{};
    \vertex (b9) at (4,-2) [label=above:$$]{};

%    \vertex (b10) at (5,0) [label=above:$x$]{};
    \vertex (b11) at (5,-1) [label=above:$$]{};
    \vertex (b12) at (5,-2) [label=above:$$]{};

    \vertex (b13) at (6,0) [label=above:$$]{};
    \vertex (b14) at (6,-1) [label=above:$$]{};

    \vertex (c1) at (2,1) [label=above:$$]{};
    \vertex (c2) at (1.5,1.5) [label=above:$$]{};
    \vertex (c3) at (2,2) [label=above:$$]{};
    \vertex (c4) at (2.5,1.5) [label=above:$v$]{};
    \vertex (c5) at (3,2) [label=above:$$]{};
    \vertex (c6) at (3.5,1.5) [label=above:$$]{};
    \vertex (c7) at (3,1) [label=above:$$]{};

    \path
    (a1) edge [-,thick] (a2)
    (a2) edge [-,thick] (a3)
    (a3) edge [-,thick] (a4)
    (a4) edge [-,thick] (a1)

    (b1) edge [-,thick] (a3)
    (b1) edge [-,thick] (a4)

    (b1) edge [-,thick] (b2)
    (b2) edge [-,thick] (b3)

    (b4) edge [-,thick] (b5)
    (b5) edge [-,thick] (b6)

    (b7) edge [-,thick] (b8)
    (b8) edge [-,thick] (b9)

    %(b10) edge [-,thick] (b11)
    (b11) edge [-,thick] (b12)

    (b13) edge [-,thick] (b14)

    (b1) edge [-,thick] (b4)
    (b4) edge [-,thick] (b7)
    %(b7) edge [-,thick] (b10)
    %(b10) edge [-,thick] (b13)

%    (b2) edge [-,thick] (b5)
%    (b5) edge [-,thick] (b8)
%    (b8) edge [-,thick] (b11)
    (b11) edge [-,thick] (b14)

    (b3) edge [-,thick] (b6)
    (b6) edge [-,thick] (b9)
    (b9) edge [-,thick] (b12)

    (c1) edge [-,thick] (b1)
    (c1) edge [-,thick] (b4)
    (c4) edge [-,thick] (b1)
    (c4) edge [-,thick] (b4)
    (c7) edge [-,thick] (b1)
    (c7) edge [-,thick] (b4)

    (c1) edge [-,thick] (c2)
    (c2) edge [-,thick] (c3)
    (c3) edge [-,thick] (c4)
    (c4) edge [-,thick] (c1)
    (c4) edge [-,thick] (c5)
    (c5) edge [-,thick] (c6)
    (c6) edge [-,thick] (c7)
    (c7) edge [-,thick] (c1)
    (c7) edge [-,thick] (c4)

	;

\end{tikzpicture}
\qquad
\begin{tikzpicture}[x=1cm, y=1cm]

    \draw (3,-2.75) node{\small $G_1^*$};

    \vertex (a1) at (0,0) [label=above:$u$]{};
    \vertex (a2) at (0,-1) [label=above:$$]{};
    \vertex (a3) at (1,-1) [label=above:$$]{};
    \vertex (a4) at (1,0) [label=above:$$]{};

    \vertex (b1) at (2,0) [label=above:$$]{};
    \vertex (b2) at (2,-1) [label=above:$$]{};
    \vertex (b3) at (2,-2) [label=above:$$]{};

    \vertex (b4) at (3,0) [label=above right:$w$]{};
    \vertex (b5) at (3,-1) [label=above:$$]{};
    \vertex (b6) at (3,-2) [label=above :$$]{};

    \vertex (b7) at (4,0) [label=above:$$]{};
    \vertex (b8) at (4,-1) [label=above:$$]{};
    \vertex (b9) at (4,-2) [label=above:$$]{};

%    \vertex (b10) at (5,0) [label=above:$x$]{};
    \vertex (b11) at (5,-1) [label=above:$$]{};
    \vertex (b12) at (5,-2) [label=above:$$]{};

    \vertex (b13) at (6,0) [label=above:$$]{};
    \vertex (b14) at (6,-1) [label=above:$$]{};

    \vertex (c1) at (2,1) [label=above:$$]{};
    \vertex (c2) at (1.5,1.5) [label=above:$$]{};
    \vertex (c3) at (2,2) [label=above:$$]{};
    \vertex (c4) at (2.5,1.5) [label=above:$v$]{};
    \vertex (c5) at (3,2) [label=above:$$]{};
    \vertex (c6) at (3.5,1.5) [label=above:$$]{};
    \vertex (c7) at (3,1) [label=above:$$]{};

    \path
    (a1) edge [-,thick] (a2)
    (a2) edge [-,thick] (a3)
    (a3) edge [-,thick] (a4)
    (a4) edge [-,thick] (a1)

    (b1) edge [-,thick] (a3)
    (b1) edge [-,thick] (a4)

    (b1) edge [-,thick] (b2)
    (b2) edge [-,thick] (b3)

    (b4) edge [-,thick] (b5)
    (b5) edge [-,thick] (b6)

    (b7) edge [-,thick] (b8)
    (b8) edge [-,thick] (b9)

    %(b10) edge [-,thick] (b11)
    (b11) edge [-,thick] (b12)

    (b13) edge [-,thick] (b14)

    (b1) edge [-,thick] (b4)
    (b4) edge [-,thick] (b7)
    %(b7) edge [-,thick] (b10)
    %(b10) edge [-,thick] (b13)

%    (b2) edge [-,thick] (b5)
%    (b5) edge [-,thick] (b8)
%    (b8) edge [-,thick] (b11)
    (b11) edge [-,thick] (b14)

    (b3) edge [-,thick] (b6)
    (b6) edge [-,thick] (b9)
    (b9) edge [-,thick] (b12)

    (c1) edge [-,thick] (b1)
    (c1) edge [-,thick] (b4)
    (c4) edge [-,thick] (b1)
    (c4) edge [-,thick] (b4)
    (c7) edge [-,thick] (b1)
    (c7) edge [-,thick] (b4)

    (c1) edge [-,thick] (c2)
    (c2) edge [-,thick] (c3)
    (c3) edge [-,thick] (c4)
    (c4) edge [-,thick] (c1)
    (c4) edge [-,thick] (c5)
    (c5) edge [-,thick] (c6)
    (c6) edge [-,thick] (c7)
    (c7) edge [-,thick] (c1)
    (c7) edge [-,thick] (c4)

    (a1) edge [-,thick] (a3)
    (c4) edge [-,thick] (c2)
    (c4) edge [-,thick] (c6)

    (b4) edge [-,thick] (b2)
    (b4) edge [-,thick] (b3)
    (b4) edge [-,thick] (b5)
    (b4) edge [-,bend right=15,thick] (b6)
    (b4) edge [-,thick] (b8)
    (b4) edge [-,thick] (b9)
	;

\end{tikzpicture}

\vskip0.5cm

\begin{tikzpicture}[x=1cm, y=1cm]

    \draw (3,-2.75) node{\small $G_2$};

    \vertex (a1) at (0,0) [label=above:$u$]{};
    \vertex (a2) at (0,-1) [label=above:$$]{};
    \vertex (a3) at (1,-1) [label=above:$$]{};
    \vertex (a4) at (1,0) [label=above:$$]{};

    \vertex (b1) at (2,0) [label=above:$$]{};
    \vertex (b2) at (2,-1) [label=above:$$]{};
    \vertex (b3) at (2,-2) [label=above:$$]{};

    \vertex (b4) at (3,0) [label=above right:$w$]{};
    \vertex (b5) at (3,-1) [label=above:$$]{};
    \vertex (b6) at (3,-2) [label=above :$$]{};

    \vertex (b7) at (4,0) [label=above:$$]{};
    \vertex (b8) at (4,-1) [label=above:$$]{};
    \vertex (b9) at (4,-2) [label=above:$$]{};

    \vertex (b10) at (5,0) [label=above:$x$]{};
    \vertex (b11) at (5,-1) [label=above:$$]{};
    \vertex (b12) at (5,-2) [label=above:$$]{};

    \vertex (b13) at (6,0) [label=above:$$]{};
    \vertex (b14) at (6,-1) [label=above:$$]{};

    \vertex (c1) at (2,1) [label=above:$$]{};
    \vertex (c2) at (1.5,1.5) [label=above:$$]{};
    \vertex (c3) at (2,2) [label=above:$$]{};
    \vertex (c4) at (2.5,1.5) [label=above:$v$]{};
    \vertex (c5) at (3,2) [label=above:$$]{};
    \vertex (c6) at (3.5,1.5) [label=above:$$]{};
    \vertex (c7) at (3,1) [label=above:$$]{};

    \path
    (a1) edge [-,thick] (a2)
    (a2) edge [-,thick] (a3)
    (a3) edge [-,thick] (a4)
    (a4) edge [-,thick] (a1)

    (b1) edge [-,thick] (a3)
    (b1) edge [-,thick] (a4)

    (b1) edge [-,thick] (b2)
    (b2) edge [-,thick] (b3)

    (b4) edge [-,thick] (b5)
    (b5) edge [-,thick] (b6)

    (b7) edge [-,thick] (b8)
    (b8) edge [-,thick] (b9)

    (b10) edge [-,thick] (b11)
    (b11) edge [-,thick] (b12)

    (b13) edge [-,thick] (b14)

    (b1) edge [-,thick] (b4)
    (b4) edge [-,thick] (b7)
    (b7) edge [-,thick] (b10)
    (b10) edge [-,thick] (b13)

%    (b2) edge [-,thick] (b5)
%    (b5) edge [-,thick] (b8)
%    (b8) edge [-,thick] (b11)
    (b11) edge [-,thick] (b14)

    (b3) edge [-,thick] (b6)
    (b6) edge [-,thick] (b9)
    (b9) edge [-,thick] (b12)

    (c1) edge [-,thick] (b1)
    (c1) edge [-,thick] (b4)
    (c4) edge [-,thick] (b1)
    (c4) edge [-,thick] (b4)
    (c7) edge [-,thick] (b1)
    (c7) edge [-,thick] (b4)

    (c1) edge [-,thick] (c2)
    (c2) edge [-,thick] (c3)
    (c3) edge [-,thick] (c4)
    (c4) edge [-,thick] (c1)
    (c4) edge [-,thick] (c5)
    (c5) edge [-,thick] (c6)
    (c6) edge [-,thick] (c7)
    (c7) edge [-,thick] (c1)
    (c7) edge [-,thick] (c4)

    (a1) edge [-,thick] (a3)
    (c4) edge [-,thick] (c2)
    (c4) edge [-,thick] (c6)

    (b4) edge [-,thick] (b2)
    (b4) edge [-,thick] (b3)
    (b4) edge [-,thick] (b5)
    (b4) edge [-,bend right=15,thick] (b6)
    (b4) edge [-,thick] (b8)
    (b4) edge [-,thick] (b9)
	;

\end{tikzpicture}
\qquad
\begin{tikzpicture}[x=1cm, y=1cm]

    \draw (3,-2.75) node{\small $G_2^*$};

    \vertex (a1) at (0,0) [label=above:$u$]{};
    \vertex (a2) at (0,-1) [label=above:$$]{};
    \vertex (a3) at (1,-1) [label=above:$$]{};
    \vertex (a4) at (1,0) [label=above:$$]{};

    \vertex (b1) at (2,0) [label=above:$$]{};
    \vertex (b2) at (2,-1) [label=above:$$]{};
    \vertex (b3) at (2,-2) [label=above:$$]{};

    \vertex (b4) at (3,0) [label=above right:$w$]{};
    \vertex (b5) at (3,-1) [label=above:$$]{};
    \vertex (b6) at (3,-2) [label=above :$$]{};

    \vertex (b7) at (4,0) [label=above:$$]{};
    \vertex (b8) at (4,-1) [label=above:$$]{};
    \vertex (b9) at (4,-2) [label=above:$$]{};

    \vertex (b10) at (5,0) [label=above:$x$]{};
    \vertex (b11) at (5,-1) [label=above:$$]{};
    \vertex (b12) at (5,-2) [label=above:$$]{};

    \vertex (b13) at (6,0) [label=above:$$]{};
    \vertex (b14) at (6,-1) [label=above:$$]{};

    \vertex (c1) at (2,1) [label=above:$$]{};
    \vertex (c2) at (1.5,1.5) [label=above:$$]{};
    \vertex (c3) at (2,2) [label=above:$$]{};
    \vertex (c4) at (2.5,1.5) [label=above:$v$]{};
    \vertex (c5) at (3,2) [label=above:$$]{};
    \vertex (c6) at (3.5,1.5) [label=above:$$]{};
    \vertex (c7) at (3,1) [label=above:$$]{};

    \path
    (a1) edge [-,thick] (a2)
    (a2) edge [-,thick] (a3)
    (a3) edge [-,thick] (a4)
    (a4) edge [-,thick] (a1)

    (b1) edge [-,thick] (a3)
    (b1) edge [-,thick] (a4)

    (b1) edge [-,thick] (b2)
    (b2) edge [-,thick] (b3)

    (b4) edge [-,thick] (b5)
    (b5) edge [-,thick] (b6)

    (b7) edge [-,thick] (b8)
    (b8) edge [-,thick] (b9)

    (b10) edge [-,thick] (b11)
    (b11) edge [-,thick] (b12)

    (b13) edge [-,thick] (b14)

    (b1) edge [-,thick] (b4)
    (b4) edge [-,thick] (b7)
    (b7) edge [-,thick] (b10)
    (b10) edge [-,thick] (b13)

%    (b2) edge [-,thick] (b5)
%    (b5) edge [-,thick] (b8)
%    (b8) edge [-,thick] (b11)
    (b11) edge [-,thick] (b14)

    (b3) edge [-,thick] (b6)
    (b6) edge [-,thick] (b9)
    (b9) edge [-,thick] (b12)

    (c1) edge [-,thick] (b1)
    (c1) edge [-,thick] (b4)
    (c4) edge [-,thick] (b1)
    (c4) edge [-,thick] (b4)
    (c7) edge [-,thick] (b1)
    (c7) edge [-,thick] (b4)

    (c1) edge [-,thick] (c2)
    (c2) edge [-,thick] (c3)
    (c3) edge [-,thick] (c4)
    (c4) edge [-,thick] (c1)
    (c4) edge [-,thick] (c5)
    (c5) edge [-,thick] (c6)
    (c6) edge [-,thick] (c7)
    (c7) edge [-,thick] (c1)
    (c7) edge [-,thick] (c4)

    (a1) edge [-,thick] (a3)
    (c4) edge [-,thick] (c2)
    (c4) edge [-,thick] (c6)

    (b4) edge [-,thick] (b2)
    (b4) edge [-,thick] (b3)
    (b4) edge [-,thick] (b5)
    (b4) edge [-,bend right=15,thick] (b6)
    (b4) edge [-,thick] (b8)
    (b4) edge [-,thick] (b9)

    (b10) edge [-,thick] (b8)
    (b10) edge [-,thick] (b9)
    (b10) edge [-,bend right=15,thick] (b12)
    (b10) edge [-,thick] (b14)
    (b10) edge [-,bend right=20,thick] (b4)
	;

\end{tikzpicture}

\end{center}
\caption{A chordalization chain
$G_0 = G_0^* <_{\{u,v,w\}} G_1 \le G_1^* <_{\{x\}}  G_2 \le G_2^*$ for a local chordalization partition $\tilde{\CCC} = (\{u,v,w\}, \{x\})$ of $G$}
\label{fig:index}
\end{figure}
\end{Ex}

In this section, we prove the following statement.
\begin{Thm}\label{thm:coloring2}
For any graph $G$,
$\chi_{DP}(G) \le \omega(G) + i(G)$.
Especially, if $G$ is non-chordal and $K_n$-minor-free, then
$\chi_{DP}(G) \le n-2 + i(G)$.
\end{Thm}
\noindent
In order to do that, we show the following theorem first.

\begin{Thm}\label{thm:newclique2}
Let $G$ be a graph,
$\tilde{\CCC} = (\CCC_1, \ldots, \CCC_\ell )$ be a local chordalization partition of a hole cover $\CCC$ of $G$,
and $G^*$ be the last graph in the chordalization chain corresponding to $\tilde{\CCC}$.
If a vertex set $K$ of $G$ forms a clique in $G^*$, then there exists a subset $\CCC^*$ of $K \cap \CCC$ such that $K \setminus \CCC^*$ is a clique in $G$ and $|\CCC^* \cap \CCC_i| \le 1$ for each $i=1,\ldots,\ell$.
\end{Thm}
\begin{proof}
Let \[
G_0 = G_0^* <_{\CCC_1} G_1 \le G_1^* <_{\CCC_2}  G_2 \le G_2^* < \cdots <_{\CCC_\ell} G_\ell \le G_\ell^* = G^*
\]
be the chordalization chain corresponding to $\tilde{\CCC}$
for graphs $G_i$ and chordal graphs $G_i^*$.
Then $\CCC_i$ is a hole cover of the graph $G_{i}$ satisfying the NC property for each $i=1$, $\ldots$, $\ell$.
For each $i=0,1,\ldots, \ell$, we add the vertices in $\bigcup_{j=i+1}^\ell \CCC_j$ to $G_i^*$ and then restore the edges in $G$ to obtain $H_i$, that is, $H_i$ is the spanning supergraph of $G$ with the edge set $E(G) \cup E(G_i^*)$.
Then, by the definitions of $G_i$ and $G^*_i$, $H_\ell=G_\ell^*$ and, for each $i=0,\ldots,\ell-1$,
%For notational convenience, $\CCC^{(i)} = \bigcup_{j=i+1}^\ell \CCC_j$ for each $i=0,\ldots,\ell-1$ and $\CCC^{(\ell)} = \emptyset$.

%By the properties (i) and (ii) of Lemma~\ref{lem:existence}, $\CCC_{i+1}$ is a hole cover of $H_i - \CCC^{(i+1)}$ satisfying the NC property and
\begin{equation*}
H_i - \bigcup_{j=i+1}^\ell \CCC_j = G^*_i, \quad H_i - \bigcup_{j=i+2}^\ell \CCC_j = G_{i+1},
\end{equation*}
and, since $G^*_{i+1}=\widehat{G_{i+1}}(\CCC_{i+1})$,
\begin{equation}\label{eq:NC}
H_{i+1} - \bigcup_{j=i+2}^\ell \CCC_j =  \widehat{\left(H_i - \bigcup_{j=i+2}^\ell \CCC_j\right)}(\CCC_{i+1}).
\end{equation}
We claim that if $L$ is a clique in $H_{i+1}$ but is not a clique in $H_{i}$, then $L \setminus \{u\}$ is a clique in $H_i$ for some vertex $u \in L \cap \CCC_{i+1}$
for each $i=0, 1,  \ldots, \ell-1$.
Suppose $L$ is a clique in $H_{i+1}$ but not a clique in $H_i$ for some $i \in \{0,1,\ldots, \ell-1\}$.
%If $L$ is a clique in $H_i$, then we are done.
%Suppose that $L$ is not a clique in $H_i$.
Then $L^* := L \setminus \bigcup_{j=i+2}^\ell \CCC_j$ is a clique in $H_{i+1} - \bigcup_{j=i+2}^\ell \CCC_j$. Since two vertices joined by an edge in $H_{i+1}$ but not in $H_i$ belong to $V(G^*_{i+1})=V(G) \setminus \bigcup_{j=i+2}^{\ell}\CCC_j$, $L^*$ is not a clique in $H_i -\bigcup_{j=i+2}^\ell \CCC_j$.
We note that \eqref{eq:NC} holds and  $\CCC_{i+1}$ is a hole cover of $G_{i+1}= H_i - \bigcup_{j=i+2}^\ell \CCC_j$ satisfying the NC property. Thus, by Theorem~\ref{thm:newclique},  there exists a vertex $u$ in $L^* \cap \CCC_{i+1}$ such that $L^* \setminus \{u\}$ is a clique in $H_i - \bigcup_{j=i+2}^\ell \CCC_j$.
For the same reason why $L^*$ is not a clique in $H_i -\bigcup_{j=i+2}^\ell \CCC_j$, $L \setminus \{u\}$ is still a clique in $H_i$.

Now we take a clique $L_0 := K$ in $H_\ell$.
For $i=0, \ldots, \ell-1$, we sequentially obtain a clique $L_{i+1}$ in $H_{\ell-i-1}$ in the following way.
If $L_i$ is a clique in $H_{\ell-i-1}$, then we let $L_{i+1} = L_i$.
If $L_i$ is not a clique in $H_{\ell-i-1}$, then, by the claim which has been proven above, there exists a vertex $u \in L_i \cap \CCC_{\ell-i}$ such that $L_i \setminus \{u\}$ is a clique in $H_{\ell-i-1}$  and we let $L_{i+1}  = L_i \setminus \{u\}$.
Let $\CCC^* = K \setminus L_\ell$. Then $K \setminus \CCC^*$ equals $L_\ell$ and so is a clique as $L_\ell$ is a clique in $H_0 = G$.
Moreover, since at most one vertex in $\CCC_{\ell-i}$ was deleted to obtain $L_{i+1}$ from $L_i$, we have $\CCC^* \subset \CCC$ and $|\CCC^* \cap \CCC_i| \le 1$ for each $i=1,\ldots,\ell$, which completes the proof.
\end{proof}

\begin{Thm}\label{thm:lco}
Let $G$ be a graph,
$\tilde{\CCC} = (\CCC_1, \ldots, \CCC_{i(G)} )$ be a local chordalization partition of a hole cover $\CCC$ of $G$,
and $G^*$ be the last graph in the chordalization chain corresponding to $\tilde{\CCC}$.
Then, for an induced subgraph $H$ of $G$,  $\omega(H^*) \le \omega(H) + i(G)$ where $H^*$ is the subgraph of $G^*$ induced by $V(H)$.
%There exists a chordal completion $G^*$ of $G$ such that $\omega(G^*) \le \omega(G) + i(G)$.
% for a chordal completion $G^*$ of $G$.
Especially, if $G$ is non-chordal and $K_n$-minor-free, then $\omega(G^{*}) \le n-2 + i(G)$.
\end{Thm}
\begin{proof}
If $G$ is chordal, then the first part of the statement is immediately true as we may take $G$ as $G^*$ and the second statement is vacuously true.
Thus we may assume $G$ is non-chordal.
Then $\ell := i(G) \ge 1$.
%By the definition of $i(G)$,
%there exist a hole cover $\CCC$ of $G$ and a local chordalization partition $\tilde{\CCC} = (\CCC_1, \ldots, \CCC_\ell)$ of $\CCC$.
Let
\[
G_0 = G_0^* <_{\CCC_1} G_1 \le G_1^* <_{\CCC_2}  G_2 \le G_2^* < \cdots <_{\CCC_\ell} G_\ell \le G_\ell^* = G^*.
\]
be the chordalization chain corresponding to $\tilde{\CCC}$.
%Since an induced subgraph of a chordal graph is chordal, $H$ is chordal and the inequality in the first statement holds for $H^*=H$.
%Suppose that $G$ is non-chordal.
Clearly $H^*$ is a chordal completion of $H$.
Let $K$ be a maximum clique of $H^*$.
Then $K$ is a clique in $G^*$.
By Theorem~\ref{thm:newclique2}, there exists a subset $\CCC^*$ of $K \cap \CCC$ such that $K \setminus \CCC^*$ is a clique in $G$ and $|\CCC^* \cap \CCC_i| \le 1$ for each $i=1,\ldots,\ell$.
Then
\[|\CCC^{*}|=\left|\CCC^{*} \cap \bigcup_{j=1}^\ell\CCC_j \right| \le \sum_{j=1}^\ell |\CCC^{*} \cap \CCC_j| \le \ell.\]
Now we note that $K \setminus \CCC^*$ is a clique in $G$, $K \subset V(H)$, and $H$ is an induced subgraph of $G$. Thus $K\setminus \CCC^*$ is a clique in $H$ and so $|K \setminus \CCC^*| \le \omega(H)$.
Therefore $\omega(H^*) = |K| \le |K \setminus \CCC^*| + |\CCC^*| \le \omega(H) + \ell$ and so the first statement is true.

%To show the first statement, take a maximum clique $K$ of $G^*$.
%If $K$ is a clique of $G$, then the inequality immediately holds.
%Suppose that $K$ is not a clique of $G$.
%By Theorem~\ref{thm:newclique2}, there exists a subset $\CCC^*$ of $K \cap \CCC$ such that $K \setminus \CCC^*$ is a clique in $G$ and $|\CCC^* \cap \CCC_i| \le 1$ for each $i=1,\ldots,\ell$.
%Then
%\[|\CCC^*|=\left|\CCC^* \cap \bigcup_{j=1}^\ell\CCC_j\right| \le \sum_{j=1}^\ell |\CCC^* \cap \CCC_j| \le \ell.\]
%Now
%\[
%\omega(G) \ge |K \setminus \CCC^*| \ge |K| - |\CCC^*| \ge \omega(G^*) - \ell
%\]
%and the first part of the statement is true.

To show the ``especially'' part, assume that $G$ is $K_n$-minor-free.
Let $Z$ be the graph with the vertex set $V(G)$ and the edge set $E(G) \cup E(G_1^*)$.
Then $\bigcup_{j=2}^\ell \CCC_j$ is a hole cover of $Z$ and $(\CCC_2, \ldots, \CCC_\ell)$ is a local chordalization partition of $\bigcup_{j=2}^\ell \CCC_j$.
By the definition of non-chordality index, $i(Z) \le \ell-1$.
Let
\[
Z_0 = Z_0^* <_{\CCC_2} Z_1 \le Z_1^* <_{\CCC_3}  Z_2 \le Z_2^* < \cdots <_{\CCC_{\ell}} Z_{\ell-1} \le Z_{\ell-1}^*
\]
be the chordalization chain corresponding to $(\CCC_2, \ldots, \CCC_\ell)$.
By the way, $Z_0 = Z_0^* = G_1^*$, $Z_i = G_{i+1}$ and $Z_i^* = G_{i+1}^*$ for $i=1,\ldots,\ell-1$.
To reach a contradiction, suppose that $Z$ has a clique $L$ of size $n$.
Then $L^* := L \setminus \bigcup_{j=2}^\ell \CCC_j$ is a clique in $G_1^*$. By the definition of $Z$, the edges in $L$ but not in $L^*$ belong to $G$.
Since $\CCC_1$ is a hole cover of $G_1$ satisfying the NC property, by Theorem~\ref{thm:knminor}, $L^*$ is a minor of $G_1$ as $G_1^* = \widehat{G_1}(\CCC_1)$.
As $G_1$ is a subgraph of $G$ and the edges in $L$ but not in $L^*$ belong to $G$, we may conclude that $L$ is a minor of $G$ with size $n$, which is a contradiction.
Therefore $Z$ is $K_n$-free and so $\omega(Z) \le n-1$.
Take a maximum clique $K$ of $G^*$.
If $K$ is a clique of $Z$, then $\omega(G^*) = |K| \le \omega(Z) \le n-1 \le n-2+i(G)$ and so the inequality holds.
Suppose that $K$ is not a clique of $Z$.
By Theorem~\ref{thm:newclique2}, there exists a subset $\CCC^{**}$ of $K \cap \left(\bigcup_{j=2}^\ell \CCC_j\right)$ such that $K \setminus \CCC^{**}$ is a clique in $Z$ and $|\CCC^{**} \cap \CCC_i| \le 1$ for each $i=2,\ldots,\ell$.
Then
\[|\CCC^{**}|=\left|\CCC^{**} \cap \bigcup_{j=2}^\ell\CCC_j \right| \le \sum_{j=2}^\ell |\CCC^{**} \cap \CCC_j| \le \ell-1.\]
Thus
\[
n-1 \ge \omega(Z) \ge |K \setminus \CCC^{**}| \ge |K| - |\CCC^{**}| \ge \omega(G^*) - (\ell-1)
\]
and the ``especially'' part is true.
\end{proof}

\emph{A proof of Theorem~\ref{thm:coloring2}.}
Take a graph $G$ and let $G^*$ be a chordal completion of $G$ given in Theorem~\ref{thm:lco}.
Then, since $G^*$ is chordal, $\chi_{DP}(G^*) = \omega(G^*)$ by ($\S$).
Thus, by Theorem~\ref{thm:lco},
\[
\chi_{DP}(G) \le \chi_{DP}(G^*) = \omega(G^*) \le \omega(G) + i(G)
\]
and, if $G$ is non-chordal and $K_n$-minor-free, then the right hand side of the second inequality above may be  replaced with $n-2+i(G)$.
\hfill {\large $\Box$}

\bigskip

By \eqref{eqn:chi,l,dp}, Theorem~\ref{thm:coloring2} gives $\chi_l(G) \le \omega(G) + i(G)$ for a graph $G$ and $\chi_l(G) \le n-2 + i(G)$ if $G$ is non-chordal and $K_n$-minor-free.
Actually, the inequality $\chi_l(G) \le \omega(G)+i(G)$ is sharp and accordingly so is the first inequality given in Theorem~\ref{thm:coloring2}.
To show it, we need the following proposition.

Given a graph $G$, we denote the independence number and the vertex cover number of $G$ by $\alpha(G)$ and $\beta(G)$, respectively.
It is well known that $\alpha(G)+\beta(G)=|V(G)|$.

\begin{Prop}\label{prop:degenerate}
Every graph $G$ is $\beta(G)$-degenerate.
\end{Prop}
\begin{proof}
Take a graph $G$.
Let $I$ be an independent set of $G$ with size $\alpha(G)$.
Take a subgraph $H$ of $G$.
Suppose $V(H) \cap I \neq \emptyset$.
Then, as $I$ is an independent set of $G$, $V(H) \cap I$ is an independent set of $H$.
Thus any vertex in $V(H) \cap I$ has degree at most $|V(H) \setminus I| \le |V(G) \setminus I| = \beta(G)$.
If $V(H) \cap I = \emptyset$, then $|V(H)| \le |V(G) \setminus I| = \beta(G)$, and so any vertex of $H$ has degree at most $\beta(G)-1$.
Hence $G$ is $\beta(G)$-degenerate.
\end{proof}

\noindent
We recall that if a graph $G$ is $k$-degenerate, then $\chi_{DP}(G) \le k+1$, from which the corollary below is immediately true.
As a matter of fact, the corollary enhances the known inequality $\chi(G) \le \beta(G)+1$ for a graph $G$.

\begin{Cor}\label{cor:independent}
For a graph $G$, $\chi_{DP}(G) \le \beta(G)+1$.
\end{Cor}

Consider a complete graph $K_n$ with $n \ge 2$.
Then $\alpha(K_n)=1$, $\beta(K_n)=|V(K_n)|-1$, and $\chi(K_n) = \chi_l(K_n) = \chi_{DP}(K_n) = |V(K_n)| = \beta(K_n)+1$.
Hence the upper bound for $\chi_{DP}(K_n)$ in Corollary~\ref{cor:independent} is sharp.

For a complete graph $K_n$ with $n \ge 2$, the inequality given in Corollary~\ref{cor:independent} is sharp even for $\chi(K_n)$ and $\chi_l(K_n)$ as we have seen above.
Yet, it is not necessarily in that way as it is known that $\beta(C_4)=2$, $\chi(C_4)=\chi_l(C_4)=2 < \beta(C_4)+1$, and $\chi_{DP}(C_4)=3=\beta(C_4)+1$.

\noindent
Now we are ready to present the following theorem, which implies that the inequality $\chi_l(G) \le \omega(G)+i(G)$ is sharp (and so $\chi_{DP}(G) \le \omega(G)+i(G)$ is sharp).

\begin{Thm}\label{thm:omega+i is sharp}
For a positive integer $s$ and a nonnegative integer $t$, there is a graph $G$ with $\chi(G)=\omega(G)=s+1$, $i(G)=t$, and $\chi_l(G)=s+t+1$.
\end{Thm}
\begin{proof}
If $t=0$, then we let $G=K_{s+1}$.
Suppose $t \ge 1$.
We may represent $t$ as the sum of $s$ nonnegative integers, that is, $t=\sum_{i=1}^{s}m_i$ for nonnegative integers $m_1, m_2, \ldots, m_s$.
Let $G$ be a graph isomorphic to $K_{1+m_1, 1+m_2, \ldots, 1+m_s, m}$ where $m = \left(s+t  \right)^{ \left(s+t  \right)}$.
Let $V_1$, $V_2$, $\ldots$, $V_s$, and $V_{s+1}$ be the partite sets of $G$ with $|V_i|=m_i+1$ for $i = 1$, $\ldots$, $s$ and $|V_{s+1}|=m$.
Now we take a vertex $v_i$ from $V_i$ for $i = 1$, $\ldots$, $s$.
Then $\CCC := \bigcup_{i=1}^{s}\left(V_i \setminus \{v_i\}\right)$ is a hole cover of $G$ with size $\sum_{i=1}^{s}m_i = t$.
Let $\CCC_1,  \CCC_2, \ldots, \CCC_t$ be all the singleton subsets of $\CCC$.
Then it is easy to check that $(\CCC_1, \CCC_2, \ldots, \CCC_t)$ is a local chordalization partition of $\CCC$.
Thus $i(G) \le t$.

On the other hand, it is obvious that $\omega(G)=s+1$.
Then, as it is easy to check that a complete multipartite graph is perfect, $$\chi(G)=\omega(G)=s+1.$$
Since $|V_{s+1}|=m$ and $|V(G) \setminus V_{s+1}| = \sum_{i=1}^{s} (1+m_i) = s+t$,
\begin{equation}\label{eqn:chi 1}
\chi_{DP}(G) \le s+t+1
\end{equation}
by Corollary~\ref{cor:independent}.
In addition, $\bigcup_{i=1}^{s}V_i$ and $V_{s+1}$ form two disjoint vertex sets of $G$ with sizes $s+t$ and $(s+t)^{(s+t)}$, respectively, so $G$ contains $K_{s+t, (s+t)^{s+t}}$ as a subgraph.
Then, from the observation made by Gravier~\cite{gravier1996hajos} that $\chi_l(K_{k, k^k}) > k$ for any positive integer $k$, we obtain
\begin{equation}\label{eqn:chi 2}
\chi_l(G) \ge s+t+1.
\end{equation}
Thus, by \eqref{eqn:chi,l,dp}, \eqref{eqn:chi 1}, and \eqref{eqn:chi 2}, $s+t+1 \le \chi_l(G)=\chi_{DP}(G) \le s+t+1$ and so $\chi_l(G)=s+t+1$.
Since $\omega(G)=s+1$, $i(G) \ge t$ by Theorem~\ref{thm:coloring2}.
As we have shown that $i(G) \le t$, we complete the proof.
\end{proof}

It is worthy of attention that Theorem~\ref{thm:omega+i is sharp} guarantees the existence of a graph $G$ with $i(G)=t$ for any nonnegative integer $t$.

We recall that $\omega(G) \le \chi(G) \le \chi_l(G) \le \chi_{DP}(G)$ for a graph $G$ and that the gaps between $\omega(G)$ and $\chi(G)$, between $\chi(G)$ and $\chi_l(G)$, and between $\chi_l(G)$ and $\chi_{DP}(G)$ can be arbitrarily large.
Yet, Theorem~\ref{thm:coloring2} tells us that the sum of those gaps cannot exceed $i(G)$.
Especially, if $G$ satisfies the NC property, then those gaps cannot exceed one and at most one of them can be one.

\subsection{%An equivalence relation on the set of vertices on holes in a graph
Making a local chordalization really local}

In this section, we devote ourselves to convincing readers that the ``local'' in our terminology ``local chordalization'' makes a sense.

%In this section, we prove the following statement.
%
%\noindent
%In order to do that, we show the following theorem first.

%\begin{Thm}
%For any graph $G$, $\chi(G) \le \omega(G)+ \min \{ |\CCC| :  \text{$\CCC$ is a  hole cover of $G$}\}$.
%\end{Thm}
%\begin{proof}
%If $G$ is a chordal, then $\chi(G) = \omega(G)$ as every chordal graph is perfect and the inequality immediately holds.
%Now take a non-chordal graph $G$ and take a hole cover $\CCC$ of $G$.
%Then $\CCC \neq \emptyset$.
%We join each vertex in $\CCC$ and each of its nonadjacent vertices in $G$, and denote the resulting graph by $G'$.
%Then the holes in $G$ are destroyed in $G'$.
%To show that $G'$ is chordal, suppose that $G'$ contains a hole $H$.
%By construction, $H$ must contain some vertex $v \in \CCC$.
%Since $v$ is adjacent to all the other vertices in $G'$, $H$ has a chord in $G'$ and we reach a contradiction.
%Thus $G'$ is chordal and so $\chi(G')=\omega(G')$.
%Let $K$ be a maximal clique of $G'$.
%Then $K \setminus \CCC$ is a clique in $G$ or an empty set, so $|K \setminus \CCC| \le \omega(G)$, which implies $\omega(G') \le \omega(G) + |\CCC|$.
%Therefore $\chi(G') \le \omega(G) + |\CCC|$.
%Since $G$ is a spanning subgraph of $G'$, $\chi(G) \le \chi(G')$ and so $\chi(G) \le \omega(G) + |\CCC|$.
%Since $\CCC$ was arbitrarily chosen, the inequality follows.
%\end{proof}

%We claim that new hole resulting from a local chordalization has its vertices on $\bigcup_{H \in \HHH(G)} V(H)$.

Let $G$ be a non-chordal graph and $\Omega(G) = \bigcup_{H \in \HHH(G)} V(H)$.
We define a relation $\sim_G$ on $\Omega(G)$ so that, for $u, v \in \Omega(G)$,
\begin{align*}
u \sim_G v \Leftrightarrow & \text{ either $u$ and $v$ are on the same hole or there exists a sequence $H_1, \ldots, H_t$ }\\ &\text{ of
 distinct holes in $\HHH(G)$ such that $u \in H_1$, $v \in H_t$, and $H_i$ and $H_{i+1}$ share } \\ &\text{ a vertex for each $i=1,\ldots,t-1$}.
\end{align*}
It is easy to see that $\sim_G$ is an equivalence relation and that, for each hole in $G$, the vertices on the hole belong to the same equivalence class.

\begin{Prop}\label{prop:local}
Let $G$ be a non-chordal graph, $H$ be a hole in $G$, and $S$ be the equivalence class under $\sim_G$ containing $V(H)$.
If adding a chord of $H$ to $G$ yields a new hole $H^*$, then $V(H^*) \subset S$.
\end{Prop}

%Fix $i \in \{1,\ldots,l\}$.
%
%Add a chord in a hole in $G[X_i]$.
%
%If a new hole $H^*$ is generated, then $V(H^*) \subset X_i$.

\begin{proof}
Since $H$ is a hole, there are two nonadjacent vertices $u$ and $v$ on $H$.
Suppose that adding the edge joining $u$ and $v$ to $G$ creates a new hole $H^*$.
Obviously $uv$ is a chord of $H$ in $G+uv$.
%Let $H^* = u \cdots x \cdots v u$ where  $x$ is a vertex in $H^*$.
Let $x$ be a vertex in $H^*$ other than $u$ and $v$.
It suffices to show $x \in S$ to complete the proof.
If $x$ is on $H$, then we are done.
Thus we may assume that $x$ is not on $H$.

\noindent
{\it Case 1}. $x$ is adjacent to an internal vertex of each of the two $(u,v)$-sections of $H$.
Since $u$, $v$, and $x$ are on the hole $H^*$ with $u$ and $v$ consecutive on $H^*$, $x$ is nonadjacent to one of $u$ and $v$ in $G+uv$.
Without loss of generality, we may assume that $x$ is nonadjacent to $v$ in $G+uv$.
Obviously $x$ is nonadjacent to $v$ in $G$.
By applying Lemma~\ref{lem:hole} for $P=\{x\}$, there exists a hole in $G$ containing $x$ and $v$.
Therefore $x \sim_G v$.
Since $v \in S$, $x \in S$.
%({\bf Note that $u, v \in S$.})

\noindent
{\it Case 2}. One of the two $(u,v)$-sections of $H$ has no internal vertex that is adjacent to $x$.
Let $R$ be such a $(u,v)$-section.
Then none of $x$ and its neighbors on $H^*$ is an internal vertex on $R$.
While traversing along the $(x,v)$-section (resp.\ $(x,u)$-section) of $H^*$ not containing $u$ (resp.\ $v$), let $y$ (resp.\ $z$) be the first vertex at which we meet $R$.
Let $Q_1$ be the $(y,z)$-section of $H^*$ containing $x$, $Q_2$ be the $(y,z)$-section of $R$, and $Q = Q_1Q_2$.
By the choices of $y$ and $z$, $Q$ is an induced cycle of $G$ containing $x$ and a vertex on $H$.
Since two neighbors of $x$ on $H^*$ are nonadjacent in $G$, $Q$ is a hole in $G$.
Since $Q$ contains $x$ and a vertex on $H$, $x \in S$.
\end{proof}

\begin{Rem}\label{rmk:local}
Let $G$ be a non-chordal graph and $\Omega(G) / {\sim_G}$ be the set of equivalence classes under $\sim_G$.
Take an equivalence class $S \in \Omega(G) / {\sim_G}$, a hole $H$ with $V(H) \subset S$, and vertices $u$ and $v$ on $H$ which are not consecutive.
%Then, by Proposition~\ref{prop:local}, the equivalence classes in $\Omega(G)/ \sim_G$ except $S$ are still  equivalence classes under $\sim_{G+uv}$.
%Furthermore, if there is another equivalence class under $\sim_{G+uv}$, then it is a subset of $S$.
Proposition~\ref{prop:local} implies that
%$G-S$ remains the same even after locally chordalizing a hole $H$ by a vertex $x$. Furthermore,
the equivalence classes in $\Omega(G)/{\sim_G}$ except $S$ are still  equivalence classes under $\sim_{G+uv}$,
and if there are other equivalence classes under $\sim_{G+uv}$,
%(since $G$ does not necessarily satisfy the NC property, ),
they are disjoint subsets of $S$.
Therefore $\Omega(G+uv) \subset \Omega(G)$.
%where $v$ is a vertex on $H$ not adjacent to $u$ in $G$.
%In this vein, we may claim that the ``local'' in our terminology ``local chordalization'' makes a sense.
%Furthermore, if there is another equivalence class under $\sim_{G+uv}$, then it is a subset of $S$.
%In the sense that locally chordalizing a hole $H$ by a vertex $v$ does not affect the graph structure outside the equivalence class containing $v$, it is natural to call the process ``local'' chordalization.
\end{Rem}

\begin{Rem}\label{rmk:local2}
Let $G$ be a non-chordal graph and $\ell = i(G)$.
By the definition of $i(G)$,
there exist a hole cover $\CCC$ of $G$ and a local chordalization partition $\tilde{\CCC} = (\CCC_1, \ldots, \CCC_\ell)$ of $\CCC$.
Let
\begin{equation}\label{eqn:chordalchain}
G_0 = G_0^* <_{\CCC_1} G_1 \le G_1^* <_{\CCC_2}  G_2 \le G_2^* < \cdots <_{\CCC_\ell} G_\ell \le G_\ell^* =: G^*
\end{equation}
be the chordalization chain corresponding to $\tilde{\CCC}$.
Let $H$ be the subgraph of $G$ induced by $\Omega(G)$.
Then, by the definition of induced subgraph, all the holes in $H$ are contained in $G$.
By the definition of $\Omega(G)$, all the holes in $G$ are contained in $H$.
Therefore $\HHH(G) = \HHH(H)$, $\Omega(G)  = \Omega(H)$, and $\CCC$ is a hole cover of $H$.
Thus the equivalence classes under $\sim_G$ are the equivalence classes under $\sim_H$.
We recall that
\begin{align}
G_0 &= G_0^* = G - \CCC; \label{eqn:rmklocal1} \\
V(G_i) &= V(G_{i-1}^*) \cup \CCC_{i}, \ E(G_i) = E(G_{i-1}^*) \cup E \left( G - \bigcup_{j=i+1}^\ell \CCC_j\right); \label{eqn:rmklocal2} \\
G_i^* &= \widehat{G_i}(\CCC_i). \notag
\end{align}
for each $i=1,\ldots,\ell$.
Let $H_0 = H_0^* = H - \CCC$.
Since $H$ is an induced subgraph of $G$, $H_0$ is an induced subgraph of $G_0$ by \eqref{eqn:rmklocal1}.
Furthermore, $G_0$, $G_0^*$, $H_0$, and $H_0^*$ are chordal and so $\HHH(G_0^*) = \HHH(G_0) = \HHH(H_0) = \HHH(H_0^*) = \emptyset$ and $\Omega(G_0^*) = \Omega(G_0) = \Omega(H_0) = \Omega(H_0^*)=\emptyset$.
Let $H_1$ be the graph defined by $V(H_{1}) = V(H_{0}^*) \cup \CCC_{1}$ and
\[
         E(H_{1}) = E(H_{0}^*) \cup E\left(H -  \bigcup_{j=2}^\ell \CCC_j\right).
\]
Since $H$ and $H_0^*$ are induced subgraphs of $G$ and $G_0^*$, respectively, $H_{1}$ is an induced subgraph of $G_{1}$ and $\HHH(H_1) \subset \HHH(G_1)$ by~\eqref{eqn:rmklocal2}.
Take a hole $\Omega_1$ in $G_1$.
Since $G_1$ is an induced subgraph of $G$,
$V(\Omega_1) \subset \Omega(G) \setminus \bigcup_{i=2}^\ell \CCC_i$.
Since $\Omega(G) = V(H)$ and $V(H) \setminus \bigcup_{i=2}^\ell \CCC_i = V(H_1)$, $V(\Omega_1) \subset  V(H_1)$.
Since $H_1$ is an induced subgraph of $G$, $\Omega_1$ is a hole in $H_1$.
Thus we have shown that $\HHH(H_1) = \HHH(G_1)$.
Hence, since $\CCC_1$ is a hole cover of $G_1$ satisfying the NC property, it is a hole cover of $H_1$ satisfying the NC property and so we obtain $\widehat{H_1}(\CCC_1) =: H_1^*$.
Since $\HHH(H_1) = \HHH(G_1)$ and $H_1$ is an induced subgraph of $G_1$, $H_1^*$ is an induced subgraph of $G_1^*$.
Let $H_2$ be the graph defined by $V(H_{2}) = V(H_{1}^*) \cup \CCC_{2}$ and
\[
         E(H_{2}) = E(H_{1}^*) \cup E\left(H -  \bigcup_{j=3}^\ell \CCC_j\right).
\]
Then  $\Omega(G)\setminus \bigcup_{i=3}^\ell \CCC_i = V(H_2)$.
Since $H$ and $H_1^*$ are induced subgraphs of $G$ and $G_1^*$, respectively, $H_{2}$ is an induced subgraph of $G_{2}$ and $\HHH(H_2) \subset \HHH(G_2)$ by~\eqref{eqn:rmklocal2}.
Take a hole $\Omega_2$ in $G_2$.
Since $G_1^*$ is chordal, $\Omega_2$ must contain a vertex $v$ in $\CCC_2$.
By the way, since $\CCC_2$ is a hole cover of $G_2$ satisfying the NC property, $\Omega_2$ contains exactly one vertex in $\CCC_2$ and so $v$ is the only vertex on $\Omega_2$ that is contained in $\CCC_2$.

Since $G$ is non-chordal, there exist a hole in $G$.
The chain given in \eqref{eqn:chordalchain} is the shortest, one of the holes in $G$ must be in $G_1$.
Thus there exists an edge in $E(G_1^*) \setminus E(G_1)$.
Take an edge $e$ in $E(G_1^*) \setminus E(G_1)$.
Then there is a hole in $G$ such that $e$ is its chord in $G+e$.
By Proposition~\ref{prop:local}, $\Omega(G+e) \subset \Omega(G)$.

If $E(G_1^*) \setminus E(G_1) = \{e\}$, then, by Proposition~\ref{prop:local},
$V(\Omega_2) \subset \Omega(G_2) \subset \Omega(G +e) \subset \Omega(G)$ and so $V(\Omega_2) \subset \Omega(G)$.
Suppose that $E(G_1^*) \setminus (E(G_1) \cup \{e\}) \neq \emptyset$ and take an edge $e'$ in $E(G_1^*) \setminus (E(G_1) \cup \{e\})$.
Then there is a hole $C$ in $G$ such that $e'$ is its chord in $G+e'$.
Now there is a hole in $G+e$ such that $e'$ is its chord in $G \cup \{e,e'\}$.
For, if $C$ is a hole in $G+e$, then it is such a hole.
Otherwise, by the definition of local chordalization, $e$ is a chord of $C$ and $e'$ is a chord of a hole from $C+e$.

%Then, by the definition of local chordalization, there is a hole in $G$ such that $e'$ is its chord in $G+e'$ and so there is a hole in $G+e$ such that $e'$ is its chord in $G \cup \{e,e'\}$.
By applying Proposition~\ref{prop:local} for $G+e$ and an edge $e'$, $\Omega(G \cup \{e,e'\}) \subset \Omega(G)$.
We may repeat this argument to conclude that $\Omega(G \cup (E(G_1^*) \setminus E(G_1))) \subset \Omega(G)$.
Since $G_2$ is an induced subgraph of $G \cup (E(G_1^*) \setminus E(G_1))$ and $\Omega_2$ is a hole in $G_2$,
\[
V(\Omega_2) \subset \Omega(G_2) \subset \Omega(G \cup (E(G_1^*) \setminus E(G_1))) \subset \Omega(G),
\]
and so $V(\Omega_2) \subset \Omega(G)$.
Therefore we have shown that $V(\Omega_2) \subset \Omega(G)$ whether or not $E(G_1^*) \setminus (E(G_1) \cup \{e\}) \neq \emptyset$.
Thus the vertices on $\Omega_2$ belong to $\Omega(G)\setminus \bigcup_{i=3}^\ell \CCC_i$.
Since $\Omega(G)\setminus \bigcup_{i=3}^\ell \CCC_i = V(H_2)$ and $H_2$ is an induced subgraph of $G_2$, $\Omega_2$ is a hole in $H_2$ and so $\HHH(G_2) \subset \HHH(H_2)$.
Thus $\HHH(G_2) = \HHH(H_2)$.
Hence, since $\CCC_2$ is a hole cover of $G_2$ satisfying the NC property, it is a hole cover of $H_2$ satisfying the NC property and so we obtain $\widehat{H_2}(\CCC_2) =: H_2^*$.
We may repeat this process to obtain $H_3, H_3^*, \ldots, H_\ell, H_\ell^*$ such that
\begin{align*}
V(H_i) &= V(H_{i-1}^*) \cup \CCC_{i}, \ E(H_i) = E(H_{i-1}^*) \cup E \left( H - \bigcup_{j=i+1}^\ell \CCC_j\right), \\ %\label{eqn:rmklocal4} \\
H_i^* &= \widehat{H_i}(\CCC_i), %\label{eqn:rmklocal5}
\end{align*}
and $\HHH(G_i) = \HHH(H_i)$ for $i=3,\ldots,\ell$.
Noting that $\HHH(G) = \HHH(H)$ and $G_\ell^*$ (resp.\  $H_\ell^*$) is a chordal completion of $G$ (resp.\ $H$), we may conclude that $i(H) \le \ell = i(G)$.

To show that $i(G) \le i(H)$, we need to introduce the chordalization chain corresponding to a local chordalization partition $\tilde{\CCC'}$ of a hole cover $\CCC'$ of $H$ terminating at $H_{i(H)}^*$.
By mimicking the previous argument constructing the chordalization chain corresponding to $\tilde{\CCC}$ for $H$, we may construct the chordalization chain corresponding to $\tilde{\CCC'}$ for $G$ to conclude $i(G) \le i(H)$.
Thus $i(G) = i(H)$ and it is sufficient to apply local chordalization process to the induced subgraph $H$ of $G$, which is a \emph{local} structure, to obtain a desired chordal completion of $G$.
In this vein, we may claim that the ``local'' in our terminology ``local chordalization'' is meaningful in another respect.
\end{Rem}

\begin{Ex}
The graph $G_2$ in Figure~\ref{fig:local} is obtained from $G_1$ by replacing the vertex $v$ of $G_1$ by the complete graph $K_n$.
Then $\Omega(G_1)=\Omega(G_2)$.
By the argument given in Remark~\ref{rmk:local2}, $i(G_1) = i(G_2)$.
Yet, the treewidths of $G_1$ and $G_2$ are $2$ and $n-1$, respectively.
\end{Ex}

\begin{figure}
\begin{center}
\begin{tikzpicture}[x=1.5cm, y=1.5cm]

    \vertex (x1) at (0,0) [label=above:$$]{};
    \vertex (x2) at (0,1) [label=above:$$]{};
    \vertex (x3) at (1,0) [label=above:$$]{};
    \vertex (v) at (1,1) [label=above:$$]{};
    \vertex (x5) at (1.6,1.6) [label=above right:$v$]{};
    \draw (0.5,-0.5) node{$G_1$};

    \path
    (x1) edge [-,thick] (x2)
    (x1) edge [-,thick] (x3)
    (v) edge [-,thick] (x3)
    (v) edge [-,thick] (x2)
    (v) edge [-,thick] (x5)
	;

    \vertex (y1) at (5,0) [label=above:$$]{};
    \vertex (y2) at (5,1) [label=above:$$]{};
    \vertex (y3) at (6,0) [label=above:$$]{};
    %\vertex (v1) at (5.85,0.85) [label=above:$$]{};
    \vertex (v1) at (6,1) [label=above:$$]{};
    \vertex (v5) at (6.6,1.6) [label=above:$$]{};

    \draw[dashed] (7.0,2.0) circle (0.6) [label=right:$K_n$]{};
    \draw (7.0,2.0) node{$K_n$};
    \draw (5.5,-0.5) node{$G_2$};
    \path

    (y1) edge [-,thick] (y2)
    (y1) edge [-,thick] (y3)
    (v1) edge [-,thick] (y3)
    (v1) edge [-,thick] (y2)
    (v1) edge [-,thick] (v5)

    (y1) edge [-,thick] (y2)
    (y1) edge [-,thick] (y3)

    (y1) edge [-,thick] (y2)
    (y1) edge [-,thick] (y3)

	;

\end{tikzpicture}
\end{center}
\caption{$\Omega(G_1)=\Omega(G_2)$, so $i(G_1)=i(G_2)$ by the argument given in Remark~\ref{rmk:local2}.}
\label{fig:local}
\end{figure}
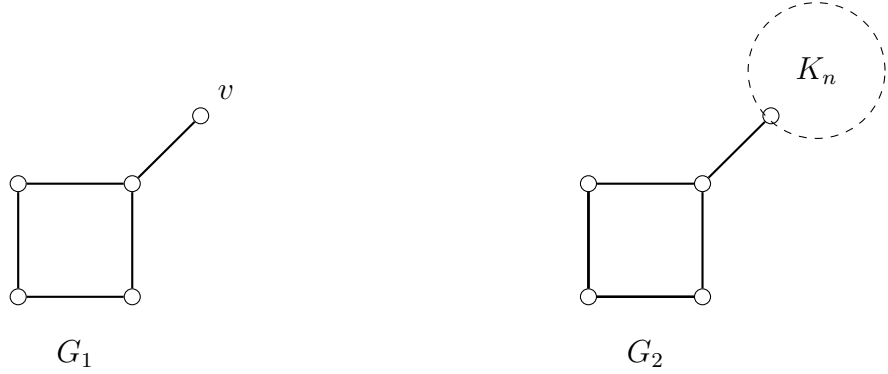

By the argument given in Remark~\ref{rmk:local}, the following proposition is true.

\begin{Prop}\label{prop:local2}
For a non-chordal graph $G$,
$i(G) = \max \{i(G[S_1]), \ldots, i(G[S_r])\}$ where $S_1, \ldots, S_r$ are the equivalence classes under $\sim_G$.
\end{Prop}
\begin{proof}
Let $G$ be a graph and $\tilde{\CCC} = (\CCC_1, \ldots, \CCC_{i(G)} )$ be a local chordalization partition of a hole cover $\CCC$ of $G$.
Then, for each $j=1, \ldots, i(G)$, $\CCC \cap S_j$ is a hole cover of $G[S_j]$.
In addition, by the argument given in Remark~\ref{rmk:local}, a subset of $\{\CCC_1 \cap S_j, \ldots, \CCC_{i(G)} \cap S_j\}$ forms a local chordalization partition of $\CCC \cap S_j$.
Thus $i(G[S_j]) \le i(G)$ for each $j=1, \ldots, i(G)$ and so $i(G) \ge \max \{i(G[S_1]), \ldots, i(G[S_r])\}$.

% $\tilde{\CCC_j} = (\CCC_1 \cap S_j, \ldots, \CCC_{i(G)}\cap S_j)$ is a local chordalization partition of $\CCC \cap S_j$ for each $j=1, \ldots, i(G)$.

Now let $\tilde{\CCC^j}=(\CCC^j_1, \ldots, \CCC^j_{i(G[S_j])})$ be a local chordalization partition of a hole cover $\CCC^j$ of $G[S_j]$ for each $j=1, \ldots, r$.
Clearly $\bigcup_{j=1}^{r}\CCC^j$ is a hole cover of $G$.
In addition, by the argument given in Remark~\ref{rmk:local}, $(\bigcup_{j=1}^{r}\CCC^j_1, \bigcup_{j=1}^{r}\CCC^j_2, \ldots, \bigcup_{j=1}^{r}\CCC^j_{\max \{i(G[S_1]), \ldots, i(G[S_r])\}})$ is a local chordalization partition of $\bigcup_{j=1}^{r}\CCC^j$ where
 $C^j_p = \emptyset$ for any $j=1, \ldots, r$ and any $p$, $i(G[S_j]) < p \le  \max \{i(G[S_1]), \ldots, i(G[S_r])\}$.
Hence $\max \{i(G[S_1]), \ldots, i(G[S_r])\} \ge i(G)$.
\end{proof}

The \emph{join}, denoted by $G_1 \vee G_2$, of two graphs $G_1$ and $G_2$ is the graph with the vertex set $V(G_1) \cup V(G_2)$ and the edge set $E(G_1) \cup E(G_2) \cup \{uv \mid u \in V(G_1) \text{ and } v \in V(G_2)\}$.
We denote by $I_m$ an empty graph with $m$ vertices.

\begin{Thm}\label{kmnfree}
Suppose that a non-chordal graph $G$ does not contain $I_m \vee K_n$ for positive integers $m \ge n$ as a subgraph and $\omega(G[\Omega(G)]) + i(G) \le m$.
Then there is a chordal completion $G^*$ of $G$ with $\omega(G^*) < m+n$.
\end{Thm}
\begin{proof}
Since $G$ is non-chordal, $i(G) \ge 1$.
Let $H$ be the subgraph of $G$ induced by $\Omega(G)$.
By the argument in Remark~\ref{rmk:local2}, $i(G) = i(H)$.
Let $H^*$ be the subgraph of $G^*$ induced by $\Omega(G)$ where $G^*$ is a chordal completion of $G$ obtained in Remark~\ref{rmk:local2}.
Then $H^*$ is a chordal completion of $H$.

Suppose to the contrary that $\omega(G^*) \ge m+n$.
Then there is a clique $K$ of size $m+n$ in $G^*$.
Clearly $K \cap \Omega(G)$ forms a clique in $G^*$.
Since $H^*$ is an induced subgraph of $G^*$, $K \cap \Omega(G)$ forms a clique in $H^*$.
By Theorem~\ref{thm:lco}, $|K \cap \Omega(G)| \le \omega(H) + i(G)$.
By the hypothesis, $|K \cap \Omega(G)| \le m$.
Since $|K| = m+n$, $|K \setminus \Omega(G)| \ge n$.
By the definition of local chordalization and Remark~\ref{rmk:local}, the end vertices of each of the edges newly added to obtain $G^*$ belong to $\Omega(G)$, so $K \setminus \Omega(G)$ still forms a clique in $G$ and each vertex in $K \cap \Omega(G)$ is adjacent to each vertex in $K \setminus \Omega(G)$ in $G$.
By moving $m-|K \cap \Omega(G)|$ vertices in $K \setminus \Omega(G)$ into $K \cap \Omega(G)$ if $|K \cap \Omega(G)| < m$, we may claim that $G$ contains $I_m \vee K_n$ as a subgraph.
This contradicts the hypothesis, so we conclude that $\omega(G^*) < m+n$.
\end{proof}

\noindent
The following corollary is an immediate consequence of Theorem~\ref{kmnfree}.

\begin{Cor}\label{cor:kmnfree}
Suppose a graph $G$ does not contain $I_m \vee K_n$ for positive integers $m \ge n$ as a subgraph and  $\omega(G[\Omega(G)]) + i(G) \le m$.
Then $\chi_{DP}(G) < m+n$.
\end{Cor}

\begin{Rem}
Since $K_{2,4}$ is non-chordal and has a hole cover which is a singleton, $i(K_{2,4}) = 1$.
Then, by Theorem~\ref{thm:coloring2}, $\chi_{DP}(K_{2,4}) \le 3$.
Yet, $\chi_{DP}(K_{2,4}) \le 5$ by Corollary~\ref{cor:kmnfree}.
Thus, for $\chi_{DP}(K_{2,4})$, Theorem~\ref{thm:coloring2} gives a better upper bound than Corollary~\ref{cor:kmnfree}.

On the other hand, for a certain graph $G$,  Corollary~\ref{cor:kmnfree} gives a better upper bound of $\chi_{DP}(G)$  than Theorem~\ref{thm:coloring2}.
To see why, consider the graph $G$ given in Figure~\ref{fig:complex}.
If $G$ contained a subgraph isomorphic to $I_8 \vee K_4$, then $G$ would have at least four vertices with degree at least $11$, which does not happen in $G$ as the two vertices common to $K_6$ and $K_{11}$ are the only vertices with degree at least $11$.
Hence $G$ does not contain $I_8 \vee K_4$ as a subgraph.

It is easy to check that $\omega(G)=11$ and $\omega(G[\Omega(G)])=5$.
The graph $G[\Omega(G)]$ is represented by using bold edges in Figure~\ref{fig:complex} and happens to be the graph given in Figure~\ref{fig:index}.
Therefore $i(G) = 2$.
Then Theorem~\ref{thm:coloring2} gives rise to $\chi_{DP}(G) \le 11+2$ while Corollary~\ref{cor:kmnfree} gives rise to $\chi_{DP}(G) \le 11$.
Furthermore, since $\omega(G)=11$, $\chi_{DP}(G)$ is actually equal to $11$.

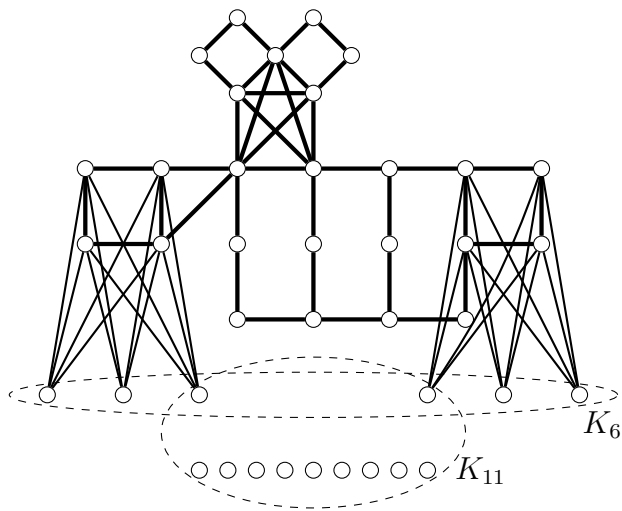
\begin{figure}[h]
\begin{center}

\begin{tikzpicture}[x=1cm, y=1cm]

    \vertex (a1) at (0,0) [label=above:$$]{};
    \vertex (a2) at (0,-1) [label=above:$$]{};
    \vertex (a3) at (1,-1) [label=above:$$]{};
    \vertex (a4) at (1,0) [label=above:$$]{};

    \vertex (b1) at (2,0) [label=above:$$]{};
    \vertex (b2) at (2,-1) [label=above:$$]{};
    \vertex (b3) at (2,-2) [label=above:$$]{};

    \vertex (b4) at (3,0) [label=above:$$]{};
    \vertex (b5) at (3,-1) [label=above:$$]{};
    \vertex (b6) at (3,-2) [label=above:$$]{};

    \vertex (b7) at (4,0) [label=above:$$]{};
    \vertex (b8) at (4,-1) [label=above:$$]{};
    \vertex (b9) at (4,-2) [label=above:$$]{};

    \vertex (b10) at (5,0) [label=above:$$]{};
    \vertex (b11) at (5,-1) [label=above:$$]{};
    \vertex (b12) at (5,-2) [label=above:$$]{};

    \vertex (b13) at (6,0) [label=above:$$]{};
    \vertex (b14) at (6,-1) [label=above:$$]{};

    \vertex (c1) at (2,1) [label=above:$$]{};
    \vertex (c2) at (1.5,1.5) [label=above:$$]{};
    \vertex (c3) at (2,2) [label=above:$$]{};
    \vertex (c4) at (2.5,1.5) [label=above:$$]{};
    \vertex (c5) at (3,2) [label=above:$$]{};
    \vertex (c6) at (3.5,1.5) [label=above:$$]{};
    \vertex (c7) at (3,1) [label=above:$$]{};

    \vertex (d1) at (-0.5,-3) [label=above:$$]{};
    \vertex (d2) at (0.5,-3) [label=above:$$]{};
    \vertex (d3) at (1.5,-3) [label=above:$$]{};
    \vertex (d4) at (4.5,-3) [label=above:$$]{};
    \vertex (d5) at (5.5,-3) [label=above:$$]{};
    \vertex (d6) at (6.5,-3) [label=above:$$]{};

    \vertex (e1) at (1.5,-4) [label=above:$$]{};
    \vertex (e2) at (1.5+1*3/8,-4) [label=above:$$]{};
    \vertex (e3) at (1.5+2*3/8,-4) [label=above:$$]{};
    \vertex (e4) at (1.5+3*3/8,-4) [label=above:$$]{};
    \vertex (e5) at (1.5+4*3/8,-4) [label=above:$$]{};
    \vertex (e6) at (1.5+5*3/8,-4) [label=above:$$]{};
    \vertex (e7) at (1.5+6*3/8,-4) [label=above:$$]{};
    \vertex (e8) at (1.5+7*3/8,-4) [label=above:$$]{};
    \vertex (e9) at (1.5+8*3/8,-4) [label=above:$$]{};

    \draw[dashed] (3,-3) ellipse (4 and 0.3);
    \draw (6.8,-3.4) node{$K_6$};
    \draw[dashed] (3,-3.5) ellipse (2 and 1);
    \draw (5.2,-4) node{$K_{11}$};

    \path
    (a1) edge [-,thick] (a2)
    (a2) edge [-,thick] (a3)
    (a3) edge [-,thick] (a4)
    (a4) edge [-,thick] (a1)

    (b1) edge [-,thick] (a3)
    (b1) edge [-,thick] (a4)

    (b1) edge [-,thick] (b2)
    (b2) edge [-,thick] (b3)

    (b4) edge [-,thick] (b5)
    (b5) edge [-,thick] (b6)

    (b7) edge [-,thick] (b8)
    (b8) edge [-,thick] (b9)

    (b10) edge [-,thick] (b11)
    (b11) edge [-,thick] (b12)

    (b13) edge [-,thick] (b14)

    (b1) edge [-,thick] (b4)
    (b4) edge [-,thick] (b7)
    (b7) edge [-,thick] (b10)
    (b10) edge [-,thick] (b13)

%    (b2) edge [-,thick] (b5)
%    (b5) edge [-,thick] (b8)
%    (b8) edge [-,thick] (b11)
    (b11) edge [-,thick] (b14)

    (b3) edge [-,thick] (b6)
    (b6) edge [-,thick] (b9)
    (b9) edge [-,thick] (b12)

    (c1) edge [-,thick] (b1)
    (c1) edge [-,thick] (b4)
    (c4) edge [-,thick] (b1)
    (c4) edge [-,thick] (b4)
    (c7) edge [-,thick] (b1)
    (c7) edge [-,thick] (b4)

    (c1) edge [-,thick] (c2)
    (c2) edge [-,thick] (c3)
    (c3) edge [-,thick] (c4)
    (c4) edge [-,thick] (c1)
    (c4) edge [-,thick] (c5)
    (c5) edge [-,thick] (c6)
    (c6) edge [-,thick] (c7)
    (c7) edge [-,thick] (c1)
    (c7) edge [-,thick] (c4)

%    (d1) edge [-,thick] (d2)
%    (d1) edge [-,bend left=15,thick] (d3)
%    (d1) edge [-,bend left=15,thick] (d4)
%    (d1) edge [-,bend left=15,thick] (d5)
%    (d1) edge [-,bend left=15,thick] (d6)
%    (d2) edge [-,thick] (d3)
%    (d2) edge [-,bend left=15,thick] (d4)
%    (d2) edge [-,bend left=15,thick] (d5)
%    (d2) edge [-,bend left=15,thick] (d6)
%    (d3) edge [-,thick] (d4)
%    (d3) edge [-,bend left=15,thick] (d5)
%    (d3) edge [-,bend left=15,thick] (d6)
%    (d4) edge [-,thick] (d5)
%    (d4) edge [-,bend left=15,thick] (d6)
%    (d5) edge [-,thick] (d6)

    (a1) edge [-,thick] (d1)
    (a2) edge [-,thick] (d1)
    (a3) edge [-,thick] (d1)
    (a4) edge [-,thick] (d1)
    (a1) edge [-,thick] (d2)
    (a2) edge [-,thick] (d2)
    (a3) edge [-,thick] (d2)
    (a4) edge [-,thick] (d2)
    (a1) edge [-,thick] (d3)
    (a2) edge [-,thick] (d3)
    (a3) edge [-,thick] (d3)
    (a4) edge [-,thick] (d3)

    (b10) edge [-,thick] (d4)
    (b11) edge [-,thick] (d4)
    (b13) edge [-,bend left=7,thick] (d4)
    (b14) edge [-,thick] (d4)
    (b10) edge [-,thick] (d5)
    (b11) edge [-,thick] (d5)
    (b13) edge [-,thick] (d5)
    (b14) edge [-,thick] (d5)
    (b10) edge [-,thick] (d6)
    (b11) edge [-,thick] (d6)
    (b13) edge [-,thick] (d6)
    (b14) edge [-,thick] (d6)

    (a1) edge [-,ultra thick] (a2)
    (a2) edge [-,ultra thick] (a3)
    (a3) edge [-,ultra thick] (a4)
    (a4) edge [-,ultra thick] (a1)

    (b1) edge [-,ultra thick] (a3)
    (b1) edge [-,ultra thick] (a4)

    (b1) edge [-,ultra thick] (b2)
    (b2) edge [-,ultra thick] (b3)

    (b4) edge [-,ultra thick] (b5)
    (b5) edge [-,ultra thick] (b6)

    (b7) edge [-,ultra thick] (b8)
    (b8) edge [-,ultra thick] (b9)

    (b10) edge [-,ultra thick] (b11)
    (b11) edge [-,ultra thick] (b12)

    (b13) edge [-,ultra thick] (b14)

    (b1) edge [-,ultra thick] (b4)
    (b4) edge [-,ultra thick] (b7)
    (b7) edge [-,ultra thick] (b10)
    (b10) edge [-,ultra thick] (b13)

%    (b2) edge [-,ultra thick] (b5)
%    (b5) edge [-,ultra thick] (b8)
%    (b8) edge [-,ultra thick] (b11)
    (b11) edge [-,ultra thick] (b14)

    (b3) edge [-,ultra thick] (b6)
    (b6) edge [-,ultra thick] (b9)
    (b9) edge [-,ultra thick] (b12)

    (c1) edge [-,ultra thick] (b1)
    (c1) edge [-,ultra thick] (b4)
    (c4) edge [-,ultra thick] (b1)
    (c4) edge [-,ultra thick] (b4)
    (c7) edge [-,ultra thick] (b1)
    (c7) edge [-,ultra thick] (b4)

    (c1) edge [-,ultra thick] (c2)
    (c2) edge [-,ultra thick] (c3)
    (c3) edge [-,ultra thick] (c4)
    (c4) edge [-,ultra thick] (c1)
    (c4) edge [-,ultra thick] (c5)
    (c5) edge [-,ultra thick] (c6)
    (c6) edge [-,ultra thick] (c7)
    (c7) edge [-,ultra thick] (c1)
    (c7) edge [-,ultra thick] (c4)

	;

\end{tikzpicture}
\end{center}
\caption{A graph $G$ which shows that Theorem~\ref{kmnfree} may be regarded as an improvement of Theorem~\ref{thm:lco}.
The vertices enclosed by a dotted ellipse form a clique.
}
\label{fig:complex}
\end{figure}

\end{Rem}

%%%%%%%%%%%%%%%%%%%%%%%%%%%
\section{New $\chi$-bounded classes}
%%%%%%%%%%%%%%%%%%%%%%%%%%%

A class $\FFF$ of graphs is said to be \emph{$\chi$-bounded} if there exists a function $f : \NN \to \RR$ such that for every graph $G \in \FFF$ and every induced subgraph $H$ of $G$,
$\chi(H) \le f(\omega(H))$.

We may extend the notion of $\chi$-boundedness as follows.
A class $\FFF$ of graphs is said to be \emph{$\chi_l$-bounded} (resp.\ \emph{$\chi_{DP}$-bounded}) if there exists a function $f : \NN \to \RR$ such that for every graph $G \in \FFF$ and every induced subgraph $H$ of $G$,
$\chi_l(H) \le f(\omega(H))$ (resp.\ $\chi_{DP}(H) \le f(\omega(H))$.

A graph $G$ is called \emph{perfect graph} if $\chi(H) = \omega(H)$ for every induced subgraph $H$ of $G$.
%It is well-known that every chordal graph is perfect.

We may also extend the notion of perfect graph as follows.
We say that a graph $G$ is \emph{list-perfect} (resp.\ \emph{DP-perfect}) if $\chi_l(H)=\omega(H)$ (resp.\ $\chi_{DP}(H)=\omega(H)$) for every induced subgraph $H$ of $G$.

We denote the class of perfect graphs, the class of list-perfect graphs, and the class of DP-perfect graphs by $\PPP$, $\PPP_l$, and $\PPP_{DP}$, respectively.

By \eqref{eqn:chi,l,dp}, a $\chi_{DP}$-bounded graph class is $\chi_l$-bounded and a $\chi_l$-bounded graph class is $\chi$-bounded.
In the proof of Theorem~\ref{thm:omega+i is sharp}, we have shown that for any positive integer $s$ and any nonnegative integer $t$, there exist a complete multipartite graph $G$ with $\omega(G)=s+1$ and $\chi_l(G)=s+t+1$, which implies that the class of complete multipartite graphs is not $\chi_l$-bounded.
Any complete multipartite graph is, however, perfect, which implies that the class of complete multipartite graphs is $\chi$-bounded.
Accordingly, a $\chi$-bounded class is not necessarily $\chi_l$-bounded.
Furthermore, $\PPP_{DP} \subset \PPP_l \subset \PPP$ by \eqref{eqn:chi,l,dp}.
Yet, $\PPP_{DP} \subsetneq \PPP_l \subsetneq \PPP$ as $K_{2, 4}$ is perfect but not list-perfect and $C_4$ is list-perfect but not DP-perfect.

Note that $\omega(C_n)=2$ and $\chi_{DP}(C_n)=3$ for even integer $n \ge 4$.
Thus no graph in $\PPP_{DP}$ contains a hole of even length.
Since a graph containing a hole of odd length is not perfect, no graph in $\PPP_{DP}$ contains a hole of odd length.
Therefore $\PPP_{DP}$ is included in the class of chordal graphs.
Thus, by ($\S$), $\PPP_{DP}$ is the class of chordal graphs.

%
%\begin{Prop}
%For a positive integer $k$, the class of graphs with a hole cover of size at most $k$ is $\chi_{DP}$-bounded.
%\end{Prop}
%\begin{proof}
%Fix a positive integer $k$.
%Let $f : \NN \to \RR$ be a function defined by $f(x)= x+k$.
%Take a graph $G$ with a hole cover $\CCC$ of size at most $k$.
%Now, for an induced subgraph $H$ of $G$, $\CCC \cap V(H)$ is a hole cover of $H$ and so
%$\chi(H) \le \omega(H) + k = f(\omega(H))$ by Corollary~\ref{cor:chibdd1}.
%Hence the statement is true.
%\end{proof}

Now we present new $\chi$-bounded classes.

\begin{Thm}\label{thm:chibound}
A family of graphs the non-chordality index of each of which does not exceed $k$ for some nonnegative integer $k$ is $\chi_{DP}$-bounded.
\end{Thm}
\begin{proof}
Take a family $\FFF$ of graphs the non-chordality index of each of which
does not exceed $k$ for a nonnegative integer $k$.
Let $f : \NN \to \RR$ be a function defined by $f(x)= x+k$.
Take a graph $G$ in $\FFF$.
Then $i(G) \le k$.
Let $H$ be an induced subgraph of $G$.
By the first part of Theorem~\ref{thm:lco},
there exists a chordal completion $H^*$ of $H$ such that $\omega(H^*) \le \omega(H) + i(G)$.
Thus $\chi_{DP}(H) \le \chi_{DP}(H^*) = \omega(H^*)  \le \omega(H) + i(G) \le f(\omega(H))$.
Hence the theorem is true.
\end{proof}

\noindent
By Remark~\ref{rmk:ncproperty}, the following corollary is immediately true.

\begin{Cor}
The class of graphs with the NC property is $\chi_{DP}$-bounded.
\end{Cor}

\section{Acknowledgement}
The second and the third authors' research was supported by
the National Research Foundation of Korea(NRF) funded by the Korea government(MEST) (No.\ NRF-2017R1E1A1A03070489) and by the Korea government(MSIP)
(No.\ 2016R1A5A1008055).
The first author's research was supported by Basic Science Research Program through the National Research Foundation of Korea(NRF) funded by the Ministry of Education(NRF-2018R1D1A1B07049150).

%(But I don't know that there is a $\chi_l$-bounded class but not a $\chi_{DP}$-bounded class.)
%(Hence $\PPP_l$ is interesting!!)


\begin{thebibliography}{1}

\bibitem{dvovrak2017correspondence}
Zden{\v{e}}k Dvo{\v{r}}{\'a}k and Luke Postle.
\newblock Correspondence coloring and its application to list-coloring planar
  graphs without cycles of lengths 4 to 8.
\newblock {\em Journal of Combinatorial Theory, Series B}, 2017.

\bibitem{erdos1979choosability}
Paul Erdos, Arthur~L Rubin, and Herbert Taylor.
\newblock Choosability in graphs.
\newblock In {\em Proc. West Coast Conf. on Combinatorics, Graph Theory and
  Computing, Congressus Numerantium}, volume~26, pages 125--157, 1979.

\bibitem{gravier1996hajos}
Sylvain Gravier.
\newblock A haj{\'o}s-like theorem for list coloring.
\newblock {\em Discrete Mathematics}, 152(1-3):299--302, 1996.

\bibitem{heggernes2006minimal}
Pinar Heggernes.
\newblock Minimal triangulations of graphs: A survey.
\newblock {\em Discrete Mathematics}, 306(3):297--317, 2006.

\bibitem{vizing1976vertex}
Vadim~G Vizing.
\newblock Vertex colorings with given colors.
\newblock {\em Diskret. Analiz}, 29:3--10, 1976.

\end{thebibliography}
\end{document}